\documentclass[11pt]{article}

\usepackage[english]{babel}
\usepackage{amssymb}
\usepackage{amsmath}
\usepackage{enumerate}
\usepackage{amsthm}
\usepackage{amsfonts}
\usepackage{listings}
\usepackage{xypic}
\usepackage{hyperref}
\usepackage{rotating}

\usepackage{MnSymbol}

\title{Model category structures on simplicial objects}
\date{May 2021}
\author{Fritz H\"ormann\\ Mathematisches Institut, Albert-Ludwigs-Universit\"at Freiburg}

\usepackage[numbers,sort&compress]{natbib}

\usepackage{color}
\usepackage{graphicx}

\usepackage{tikz}

\newcommand{\comment}[1]{}

\newtheorem{SATZ}{Theorem}[section]

\newtheorem{LEMMA}[SATZ]{Lemma}
\newtheorem{DEF}[SATZ]{Definition}
\newtheorem{PROP}[SATZ]{Proposition}

\newtheorem{BEISPIEL}[SATZ]{Example}

\newtheorem{KOR}[SATZ]{Corollary}

\newtheorem{BEM}[SATZ]{Remark}

\newtheoremstyle{bare}        
  {}            
  {}            
  {\normalfont}                 
  {}                            
  {\bfseries}                   
  {}                            
  {.0em}                           
  {\thmnumber{#2}#1. \thmnote{\normalfont\textsc{(#3)}} } 

\theoremstyle{bare}
\newtheorem{PAR}[SATZ]{}

\newcommand{\N}{ \mathbb{N} }

\newcommand{\DD}{ \mathbb{D} }

\DeclareMathOperator{\hocolim}{hocolim}
\DeclareMathOperator{\hocolimshort}{c}
\DeclareMathOperator{\holim}{holim}

\DeclareMathOperator{\proj}{proj}
\DeclareMathOperator{\inj}{inj}
\DeclareMathOperator{\eff}{eff}
\DeclareMathOperator{\spl}{split}

\DeclareMathOperator{\Fun}{Fun}
\DeclareMathOperator{\Fib}{Fib}
\DeclareMathOperator{\Cof}{Cof}
\DeclareMathOperator{\Tfib}{Tfib}
\DeclareMathOperator{\Tcof}{Tcof}
\DeclareMathOperator{\Ex}{Ex}
\DeclareMathOperator{\Mor}{Mor}

\DeclareMathOperator{\colim}{colim}

\DeclareMathOperator{\id}{id}

\DeclareMathOperator{\coker}{coker}
\DeclareMathOperator{\Hom}{Hom}

\DeclareMathOperator{\Dia}{Dia}

\DeclareMathOperator{\op}{op}

\DeclareMathOperator{\pr}{pr}
\DeclareMathOperator{\Cat}{Cat}

\newcommand*\ccircled[1]{\tikz[baseline=(char.base)]{
            \node[shape=circle,draw,inner sep=2pt] (char) {#1};}}

\newcommand{\tw}[1]{ {{}^{\downarrow \uparrow} #1 }}

\begin{document}

\maketitle



\section*{Abstract}
A general method for lifting weak factorization systems in a category $\mathcal{S}$ to model category structures on simplicial objects in $\mathcal{S}$ is described, analogously to the
lifting of cotorsion pairs in Abelian categories to model category structures on chain complexes. 
This generalizes Quillen's original treatment of (projective) model category structures on simplicial objects. As a new application we show that there is
always a model category structure on simplicial objects in the coproduct completion of any (even large) category $\mathcal{S}$ with finite limits, which --- if $\mathcal{S}$ is small --- defines the same
homotopy theory as any global model category of simplicial pre-sheaves on $\mathcal{S}$. If $\mathcal{S}$ is extensive, a similar model category structure exists on simplicial objects in $\mathcal{S}$ itself. In any case the associated left derivator exists on all diagrams despite the lack of colimits in the original category.

\tableofcontents

\section{Introduction}

It is well known that there exist several (global\footnote{i.e.\@ in which weak equivalences are the section-wise weak equivalences}) simplicial model category structures on the category $\mathcal{SET}^{\Delta^{\op} \times \mathcal{S}^{\op}}$ of simplicial pre-sheaves on a small category $\mathcal{S}$. Sometimes it is desirable to
have similar structures also on simplicial objects in $\mathcal{S}$ or at least in $\mathcal{S}^\amalg$ (free coproduct completion). In this article we show that a 
corresponding simplicial model category structure on $\mathcal{S}^{\amalg, \Delta^{\op}}$ exists regardless of whether $\mathcal{S}$ is small. 
For this it is sufficient that $\mathcal{S}$ has finite limits. 
The corresponding homotopy theories (as $\infty$-categories, categories with weak equivalences, or derivators) are
equivalent to those of simplicial pre-sheaves if $\mathcal{S}$ is small. 
There exists also a similar structure on $\mathcal{S}^{\Delta^{\op}}$ itself (without involving the coproduct completion) if $\mathcal{S}$ is (infinitary) extensive and thus has all coproducts and every object is a coproduct of $\N$-small objects. It is equivalent to the left Bousfield localization at the coproduct covers of the model category $\mathcal{S}^{\amalg, \Delta^{\op}}$. 

The author thanks Nicola Gambino for pointing out  that a similar result has been obtained independently in \cite{GHSS21}.

The article starts with a general discussion of model category structures on simplicial objects in a category, reinterpreting and slightly generalizing  Quillen's original treatment \cite{Qui67}. The method is a general machinery to ``lift'' a given weak factorization system on $\mathcal{S}$ to a simplicial model category structure on $\mathcal{S}^{\Delta^{\op}}$, similarly to the ``lifting'' of cotorsion pairs in an Abelian category to model category structures on chain complexes \cite{Hov07}. We concentrate mostly on the easier situation of ``projective cases'', i.e.\@ roughly those cases in which the left class $\mathcal{L}$ of the underlying weak factorization system is generated by morphisms of the form $\emptyset \rightarrow P$, where the objects $P$ are (in a subclass of) the projectives for the weak factorization system.
The aforementioned model category structures on $\mathcal{S}^{\amalg, \Delta^{\op}}$, and $\mathcal{S}^{\Delta^{\op}}$, respectively, arise from the weak factorization systems
$(\mathcal{L}_{\proj,\spl}, \mathcal{R}_{\proj,\spl})$ in which $\mathcal{L}_{\proj,\spl}$ is the class of coproduct injections and $\mathcal{R}_{\proj,\spl}$ is the class of split surjections. 
 
In the end we show that, despite the lack of all colimits, all homotopy colimits, and all homotopically finite homotopy limits exist.
More generally, we get an associated left derivator on {\em all} small categories, and a right derivator on homotopically finite diagrams.

\section{Extensive categories}
Extensive categories are categories in which coproducts and pull-backs interact nicely. 

\begin{DEF}\label{DEFEXTENSIVE}
A category is {\bf (infinitary) extensive} if it has all coproducts and pull-backs of coproduct injections, and coproducts are disjoint and stable under pull-back. 
\end{DEF}
We will usually drop the adjective infinitary as we will not consider finitely extensive categories. 

\begin{DEF}\label{DEFFREECOPRODUCTCOMPLETION}
Let $\mathcal{S}$ be a category.
The {\bf free coproduct completion} $\mathcal{S}^\amalg$ is defined as the category of pairs $(X, (S_x)_{x \in X})$ where $X$ is a set and the $S_x$ are objects of $\mathcal{S}$ with
obvious morphisms. Denote by $U$ the full embedding $S \mapsto (\cdot, (S))$ from $\mathcal{S}$ into its coproduct completion. 
\end{DEF}

We have the following:
\begin{PROP}\label{PROPEXTENSIVE}
\begin{enumerate}
\item If  limits of some shape $I \in \Cat$ exist in $\mathcal{S}$ then they do also exist in $\mathcal{S}^\amalg$.
\item The free coproduct completion is extensive.
\item If $\mathcal{S}$ has all coproducts then $U$ has a left adjoint $\amalg$ given by 
\[ (X, (S_x)_{x \in X}) \mapsto \coprod_{x \in X} S_x. \]
\item If $\mathcal{S}$ is extensive we have
\[ \Hom_{\mathcal{S}}(S, \amalg T) = \colim_{\substack{S \cong \amalg S' 
}} \Hom_{\mathcal{S}^\amalg}(S', T) \]
\item $\mathcal{S}$ is extensive if and only if $\mathcal{S}$ has all coproducts, pull-backs of coproduct injections, and $\amalg$ commutes with finite limits. 
\end{enumerate}
\end{PROP}
\begin{proof}
1.\@--3.\@ are easy. 

4.\@
The colimit goes over the category $(\mathcal{S}^{\amalg} \times_{/^\sim \mathcal{S}} \{ S \})$ of pairs $(S', \mu: \amalg S' \cong S)$.  
There is a map from the right hand side to the left hand side given by the maps $\Hom(S', T) \rightarrow \Hom(\amalg S', \amalg T)$ and composition with $\mu$.
Since $\mathcal{S}$ is extensive, for any morphism $f: S \rightarrow \amalg T$, we get an induced decomposition $S \cong \amalg S'$ 
 such that $f = \amalg f'$
for a unique $f': S' \rightarrow T$. 
One checks that these maps are inverse to each other. 

5.\@ For an extensive category the category $(\mathcal{S}^{\amalg} \times_{/^\sim \mathcal{S}} \{ S \})$ is cofiltered. Hence for a diagram $F: J \rightarrow \mathcal{S}^{\amalg}$ we have
\[ \Hom(S, \lim_j \amalg F(j)) = \lim_j \Hom(S, \amalg F(j)) = \lim_j \colim_{S \cong \amalg S'} \Hom(S', F(j))  \]
\[  =\colim_{S \cong \amalg S'} \lim_j \Hom(S', F(j)) =  \colim_{S \cong \amalg S'}  \Hom(S', \lim_j F(j))  = \Hom(S, \amalg \lim_j F(j))  \] 
because filtered colimits commute with finite limits in $\mathcal{SET}$. We conclude by Yoneda's Lemma. 
The converse follows exploiting the commutation with suitable pull-backs. 
\end{proof}

\section{Weak factorization systems}

\begin{PAR}
Let $\mathcal{A},  \mathcal{B},  \mathcal{C}$ be categories and 
$F: \mathcal{A} \times \mathcal{B} \rightarrow \mathcal{C}$ a functor. 
For morphisms $f: X \rightarrow Y$ in $\mathcal{A}$ and $g: Z \rightarrow W$ in $\mathcal{B}$ we write
\[ \xymatrix{  F(X,W) \amalg_{F(X,Z)} F(Y,W) \ar[rr]^-{\boxplus F(f, g)}& &  F(W, Z)}  \]
and
\[ \xymatrix{ F(X, Y) \ar[rr]^-{\boxdot F(f, g)} & &  F(Y, Z) \times_{F(Y, W)} F(X, Z) }   \]
and do so similarly for functors of more variables, if the corresponding push-out, resp.\@ pull-back exists. 
\end{PAR}

\begin{DEF}
Let $\mathcal{C}$ be a category. We say that a morphism $\lambda$ has the left lifting property w.r.t.\@ a morphism $\rho$, or equivalently, 
that $\rho$ has the right lifting property w.r.t.\@ $\lambda$, denoted
\[ \lambda \Box \rho \]
if for any commutative diagram
\[ \xymatrix{ A \ar[r] \ar[d]_{\lambda}  & B \ar[d]^\rho \\ C \ar@{.>}[ru]^\sigma \ar[r] & D } \]
there is a lift $\sigma$ making the triangles commute or, in other words, if
$\boxdot \Hom(\lambda^{\op}, \rho)$
is surjective. 
For a subclass $\mathcal{M}$ of morphisms we denote
\[ {}^\Box \mathcal{M} := \{ f \in \Mor(\mathcal{C})\  |\ f \Box g \ \forall g \in \mathcal{M}  \}, \]
\[  \mathcal{M}^\Box := \{ g \in \Mor(\mathcal{C})\  |\ f \Box g \ \forall f \in \mathcal{M}  \}. \]
\end{DEF}

\begin{LEMMA}Let $\mathcal{C}_1, \dots, \mathcal{C}_n, \mathcal{D}_1, \dots, \mathcal{D}_m$ and $\mathcal{E}$ be categories, let 
 $F$ be a functor $\mathcal{C}_1 \times  \cdots \times \mathcal{C}_n \rightarrow \mathcal{D}_i$, and let $G$ be a functor $\mathcal{D}_1 \times  \cdots \times   \mathcal{D}_m \rightarrow \mathcal{E}$. Then
\[ \boxplus G(g_1, \dots, \boxplus F(f_1, \dots, f_n), \dots, g_m) = \boxplus (G \circ^i F) (g_1, \dots, f_1, \dots, f_n, \dots, g_m).  \]
\end{LEMMA}

\begin{PAR}\label{PARADJUNCTION}
If $F$ is a left adjoint of two variables with right adjoints $G_1, G_2$ then
applying the above to the compositions $\Hom(F(-,-),-) = \Hom(-,G_1(-,-)) = \Hom(-,G_2(-,-))$, we get
\[ \boxdot \Hom((\boxplus F(f, g))^{\op}, h) = \boxdot \Hom(f^{\op}, \boxdot G_1(g^{\op}, h)) = \boxdot \Hom(g^{\op}, \boxdot G_2(f^{\op}, h))  \]
In particular 
\[ \boxplus F(f, g) \Box h  \ \Leftrightarrow \  f \Box  \boxdot G_1(g^{\op}, h) \  \Leftrightarrow \ g \Box \boxdot G_2(f^{\op}, h). \] 
\end{PAR}

\begin{DEF}
A {\bf weak factorization system} on a category $\mathcal{C}$ consists of two classes $\mathcal{L}$ and $\mathcal{R}$ of morphisms of $\mathcal{C}$ such that
\begin{enumerate}
\item $\lambda \Box \rho$ for all $\lambda \in \mathcal{L}$ and $\rho \in \mathcal{R}$.
 \item 
Every morphism $f$ in $\mathcal{C}$ has a  factorization 
\[ f = \rho \circ \lambda \]
with $\lambda \in \mathcal{L}$ and $\rho \in \mathcal{R}$.
\item The following equivalent conditions (assuming 1.\@ and 2.) hold:
\begin{enumerate}
\item  $\mathcal{L}$ and $\mathcal{R}$ are closed under retracts;
\item $\mathcal{L} = {}^\Box \mathcal{R}$ and $\mathcal{R} = \mathcal{L}^\Box$.  
\end{enumerate}
\end{enumerate}
\end{DEF}

\begin{PAR}
An object $P$ which has the left lifting property w.r.t.\@ $\mathcal{R}$, i.e.\@ such that every diagram
\[ \xymatrix{& X \ar[d]^{\in \mathcal{R}} \\
P \ar[r] \ar@{.>}[ru] & Y  }\]
has a lifting as indicated, is called a {\bf projective} w.r.t.\@ $\mathcal{R}$, or also a projective w.r.t.\@ the weak factorization system. 
If $\mathcal{C}$ has an initial object $\emptyset$ the projectives are the objects $P$ such that $\emptyset \rightarrow P$ is in $\mathcal{L}$. 
Similarly {\bf injectives} are defined. 
\end{PAR}

\begin{PAR}We say that a class $L \subset \mathcal{L}$ {\bf generates} $\mathcal{L}$ if $\mathcal{R} = L^\Box$. 
A weak factorization system is called {\bf left generated} if $\mathcal{C}$ is cocomplete and there is a {\em{}set} $L \subset \mathcal{L}$ consisting of morphisms with $\kappa$-small codomain which generates $\mathcal{L}$.
By the small object argument any such set of morphisms with $\kappa$-small codomain generates a weak factorization system with $\mathcal{L} = {}^\Box(L^\Box)$ and $\mathcal{R} = L^\Box$. 
Note that the class ${}^\Box(L^\Box)$ is then equivalently the class of retracts of transfinite compositions of pushouts of elements in $L$. 
\end{PAR}

\begin{PAR}\label{PARWFSFUNCTORIALITY}
Weak factorization systems $(\mathcal{C}, \mathcal{L}, \mathcal{R})$ form a 2-multicategory in which multimorphisms 
\[ \Hom((\mathcal{C}_1, \mathcal{L}_1, \mathcal{R}_1), \dots,  (\mathcal{C}_n, \mathcal{L}_n, \mathcal{R}_n); (\mathcal{C}_0, \mathcal{L}_0, \mathcal{R}_0))   \]
are adjunctions in several variables $F \dashv G_1, \dots, G_n$ such that the equivalent conditions (by \ref{PARADJUNCTION})
\begin{enumerate}
\item[(a)] $\boxplus F(\mathcal{L}_1, \dots, \mathcal{L}_n) \subseteq \mathcal{L}_0$
\item[(b$_i$)] $\boxdot  G_i(\mathcal{L}_1, \overset{\widehat{i}}{\dots}, \mathcal{L}_n; \mathcal{R}_0) \subseteq \mathcal{R}_i$
\end{enumerate}
hold and such that 0-ary morphisms $\Hom(;(\mathcal{C}, \mathcal{L}, \mathcal{R}))$ are just objects in $\mathcal{L}$.

Axiom 1.\@ of weak factorization system implies that $\Hom: \mathcal{C}^{\op} \times \mathcal{C} \rightarrow \mathcal{SET}$ is such a functor (right variant), where
in $\mathcal{SET}$ is equipped with the weak factorization system (injections, surjections). It is part of a multimorphism
$\mathcal{SET}, \mathcal{C} \rightarrow \mathcal{C}$
 of weak factorization systems in the sense above if and only if $\mathcal{C}$ has all products and coproducts. 
\end{PAR}

\begin{PAR}[Transport of structure]\label{PARWFSTRANSPORT}
Let $(\mathcal{C}_i, \mathcal{L}_i, \mathcal{R}_i)$ for $i=1, \dots, n$ be weak factorization systems and
\[ F: \mathcal{C}_1 \times \dots \times \mathcal{C}_n \rightarrow \mathcal{C}_0    \]
be a functor with right adjoints $G_1, \dots, G_n$. We can transport the structures to $\mathcal{C}_0$ setting: 
\[ \mathcal{L}_0 := {}^\Box( \boxplus F(\mathcal{L}_1, \dots, \mathcal{L}_n)^\Box). \]
 Equivalently, we have 
 \[ \mathcal{R}_0 := \{ \rho \ | \  \boxdot G_i(\mathcal{L}_1, \overset{\widehat{i}}{\dots}, \mathcal{L}_n; \rho) \subset \mathcal{R}_i \} \] 
  for any $i$. 
 It is not clear whether the structure has factorizations and thus is a weak factorization system. 
 For $n=1$ there is a dual operation in which the structure is transported along the right adjoint. 
 \end{PAR}
  
\begin{PAR}If $(\mathcal{L},\mathcal{R})$ is a weak factorization system then  
it follows that $\mathcal{L}$ is closed under (transfinite) compositions and pushouts (as soon as they exist) and dually for $\mathcal{R}$. 
\end{PAR}

In a locally presentable category any {\em set} of morphisms generates a weak factorization system.

\begin{PAR}[Basic examples]
In the category of sets we have the weak factorization system 
\begin{center}
\begin{tabular}{rllcrll}
$\mathcal{L}$ &=& injective maps, & & $\mathcal{R}$ &=& surjective maps.
\end{tabular}
\end{center}

This generalizes in (at least) 6 different ways to a general category:

\begin{center}
\begin{tabular}{rllrll}
$\mathcal{L}_{\inj}$ &=& monomorphisms & $\mathcal{R}_{\inj}$ &=& $(\mathcal{L}_{\inj})^{\Box}$ \\
$\mathcal{L}_{\inj,\eff}$ &=& effective monomorphisms &  $\mathcal{R}_{\inj, \eff}$ &=& retracts of $A \times B \rightarrow A$ w.\@ $B$ inj.  \\
$\mathcal{L}_{\inj,\spl}$ &=& split monomorphisms & $\mathcal{R}_{\inj, \spl}$ &=& retracts of $A \times B \rightarrow A$ \\
\\
\hline
\\
$\mathcal{L}_{\proj}$ &=& ${}^\Box (\mathcal{R}_{\proj})$ & $\mathcal{R}_{\proj}$ &=& epimorphisms \\
$\mathcal{L}_{\proj,\eff}$ &=& retracts of $A  \rightarrow A \amalg B$ w.\@ $B$ proj. & $\mathcal{R}_{\proj, \eff}$ &=& effective epimorphisms \\
$\mathcal{L}_{\proj,\spl}$ &=& retracts of $A \rightarrow A \amalg B$  & $\mathcal{R}_{\proj, \spl}$ &=& split epimorphisms 
\end{tabular}
\end{center}
These structures are in general not necessarily weak factorization systems because it is not clear that a factorization exists. 
We now collect some sufficient conditions for this: 
\end{PAR}

\begin{DEF}
We call a class of morphisms $\mathcal{R}$ {\bf right-saturated} if in a commutative diagram
\[ \xymatrix{ X \ar[r]^-{can} \ar[rd]_{f} & X \amalg Y \ar[d]^{f \amalg g} \\ 
& Z}\]
$f \in \mathcal{R}$ implies $f \amalg g \in \mathcal{R}$. Similarly, $\mathcal{L}$ is called {\bf left-saturated} if $f$ in $\mathcal{L}$ implies $f \times g \in \mathcal{L}$:
\[ \xymatrix{ X \ar@{<-}[r]^-{can} \ar@{<-}[rd]_{f} & X \times Y \ar@{<-}[d]^{f \times g} \\ 
& Z}\]
\end{DEF}

\begin{LEMMA}\label{LEMMAGENWFS}
Let $\mathcal{R}$ be a class of morphisms which is right-saturated and closed under retracts.
Assume that $\mathcal{C}$ has finite coproducts and enough projectives (i.e.\@ for any object $X$ there is a projective $P$ w.r.t.\@ $\mathcal{R}$ and a morphism $P \rightarrow X$ in $\mathcal{R}$). Then setting 
\[ \mathcal{L}:=\{\text{ retracts of } X \rightarrow X \amalg P \text{ with $P$ projective } \}\] 
the pair $(\mathcal{L}, \mathcal{R})$ is a weak factorization system
and $\mathcal{L}$ is generated by the class $\{\emptyset \rightarrow P\}$ with $P$ projective  w.r.t.\@ $\mathcal{R}$. We call such weak factorization systems {\bf of projective type}. 
 
There is an obvious dual statement whose formulation we leave to the reader. 
\end{LEMMA}
\begin{proof}
By definition of projective we have $\mathcal{L} \Box \mathcal{R}$. Furthermore, for a morphism $X \rightarrow Y$ choose a projective $P$ with morphism $P \rightarrow Y$ in $\mathcal{R}$. Then $X \rightarrow X \amalg P \rightarrow Y$ is a valid factorization by assumption. 
The additional statement is clear. 
\end{proof}

\begin{PROP}
Let $\mathcal{S}$ be a category with finite coproducts.  Then 
$(\mathcal{L}_{\proj,\spl}, \mathcal{R}_{\proj,\spl})$ is a weak factorization system. 

If $\mathcal{S}$ is extensive, then $\mathcal{L}_{\proj,\spl}$ is just the class of morphisms {\em isomorphic} to $X \rightarrow X \amalg Y$. 

Dually, let $\mathcal{S}$ be a category with finite products. Then $(\mathcal{L}_{\inj,\spl}, \mathcal{R}_{\inj,\spl})$ is a weak factorization system.
\end{PROP}
\begin{proof}
Every object is projective w.r.t.\@ $\mathcal{R}_{\proj,\spl}$ and obviously $\mathcal{R}_{\proj,\spl}$ is saturated and closed under retract. Hence $(\mathcal{L}_{\proj,\spl}, \mathcal{R}_{\proj,\spl})$ is a weak factorization system by Lemma~\ref{LEMMAGENWFS}. 

For the additional statement consider a retract diagram 
\[ \xymatrix{ A \ar[r] \ar[d]^\alpha & B \ar@{^{(}->}[d] \ar[r] & A \ar[d]^{\alpha}  \\  C \ar[r]_-{\sigma} & B \coprod E \ar[r]_-{\pi} & C } \]
If $\mathcal{C}$ is extensive, we get an isomorphic left square of the form
\[ \xymatrix{ \alpha^{-1}\sigma^{-1}B \coprod \alpha^{-1}\sigma^{-1}E \ar[r] \ar[r] \ar[d]^\alpha & B \ar@{^{(}->}[d]  \\  \sigma^{-1}B \amalg \sigma^{-1}E \ar[r] & B \amalg E  } \]
and the map $\alpha^{-1}\sigma^{-1}E \rightarrow E$ has to factor through $\emptyset$ and hence $\alpha^{-1}\sigma^{-1}E \cong \emptyset$.
Hence $\alpha$ is in $\mathcal{L}_{\proj,\spl}$.
\end{proof}

\begin{PROP}
Let $\mathcal{S}$ be a category with finite coproducts, pullbacks, and enough projectives w.r.t.\@ $\mathcal{R}_{\proj,\eff}$.
Then $(\mathcal{L}_{\proj,\eff}, \mathcal{R}_{\proj,\eff})$ is a weak factorization system on $\mathcal{S}$.

Dually, let  $\mathcal{S}$ be a category with finite products, pushouts, and enough injectives. Then $(\mathcal{L}_{\inj,\eff}, \mathcal{R}_{\inj,\eff})$ is a weak factorization system on $\mathcal{S}$. 
\end{PROP}
\begin{proof}
By Lemma~\ref{LEMMAGENWFS}, we have to show that effective epimorphisms are closed under retracts and
that $X \amalg P \rightarrow Y$ is an effective epimorphism if $P \rightarrow Y$ is such. 

First we show that any pullback of an effective epimorphism is a (possibly not yet effective) epimorphism. 
Consider a diagram 
\[ \xymatrix{ & X \times_Y Z \ar[r] \ar[d] & X \ar[d]^{\alpha} \\ P \ar[r] \ar@{.>}[ru] & Z \ar[r] & Y  }\]
in which $\alpha$ is an effective epimorphism and $P \rightarrow Z$ is a projective cover.
Then a factorization indicated by the dotted arrow exists. Since $P \rightarrow Z$ is an epimorphism, also $X \times_Y Z \rightarrow Z$ has to be an epimorphism. 
Let $X \amalg P \rightarrow Y$ be a morphism whose component $P \rightarrow Y$ effective epimorphism. We have to see that it is again an effective epimorphism.
Consider a diagram
\[ \xymatrix{  
(X \amalg P) \times_Y (X \amalg P) \ar@<5pt>[d] \ar@<-5pt>[d] \\ X \amalg P \ar[r]^{\beta = (\beta_X, \beta_P)} \ar[d]  & C \\ Y \ar@{.>}[ru]^\alpha
} \]
in which the compositions of $\beta$ with the two projections are equal. 
We have to show that $\beta$ factorizes through a morphism $\alpha$ as indicated by the dotted arrow. 
Restriction to $P \times_Y P$ shows that there is a morphism $\alpha: Y \rightarrow C$ such that in
\[ \xymatrix{  
P \times_Y X \ar[r]^{\pr_2} \ar[d]^{\pr_1} & X \ar[r]^f \ar[d]_{\beta_X} & Y \ar[ld]^\alpha \\
P \ar[r]^{\beta_P} \ar[d] &C \\
Y \ar[ru]_\alpha 
} \]
the {\em lower} triangle commutes. Since the two compositions $P \times_Y X \rightarrow C$ are also the same and $P \times_Y X \rightarrow X$ is again an epimorphism, also the upper triangle commutes and therefore $\beta$ factors through $\alpha$. 

Next we show that effective epimorphisms are closed under retracts. Consider a retract and a morphism $\alpha: A \rightarrow C$ with coequalizer condition. 
\[ \xymatrix{ 
A \times_{A'} A \ar@<2pt>[d] \ar@<-2pt>[d] \ar[r] & 
B \times_{B'} B \ar@<2pt>[d] \ar@<-2pt>[d] \ar[r] & 
A \times_{A'} A \ar@<2pt>[d] \ar@<-2pt>[d] \ar[dr] \\ 
A \ar[r]^\sigma \ar[d]_f & B \ar[d]^{\pi} \ar[r]^\tau & A \ar[d]^{f} \ar[r]^\alpha & C  \\  A' \ar[r]_{\sigma'} & B' \ar[r]_{\tau'} \ar@{.>}[rru]^(.4)\beta & A' } \]
Using that $\pi$ is an effective epimorphism, contemplating the diagram one sees that a lift $\beta$ exists satisfying $\beta \pi = \alpha \tau$. 
But then $\beta \sigma' f = \beta \pi \sigma = \alpha \tau \sigma = \alpha$. Hence $\beta \sigma'$ is the required lift  from $A'$.  
\end{proof}

The proposition shows in particular that under the given hypotheses effective epimorphisms enjoy all the properties (e.g.\@ closure under composition, retract, pull-back) of the class $\mathcal{R}$ of a weak factorization system. 

\begin{KOR}
Let $I$ be a small category. In the category of presheaves $\mathcal{SET}^{I^{\op}}$ the class $\mathcal{L}_{\proj,\eff}$ consists of morphisms of the form $F \rightarrow F \amalg \coprod X_i$ where the $X_i$ are retracts of representables,
and  $\mathcal{R}_{\proj,\eff}$ consists of the element-wise surjections. 

If in $I$ idempotents split (e.g.\@ if $I$ has finite limits {\em or} colimits) then the projective objects are the coproducts of representables and $\mathcal{L}_{\proj,\eff}$ consists of morphisms of the form $F \rightarrow F \amalg \coprod h_{X_i}$.

In this case the structure is also the same as $(\mathcal{L}_{\proj}, \mathcal{R}_{\proj})$ but not the same as $(\mathcal{L}_{\proj,\spl}, \mathcal{R}_{\inj,\spl})$.
\end{KOR}
\begin{proof}
This structure is a special case of the previous because the effective epimorphisms are the element-wise surjections in this case. (Every epimorphism in a topos is effective).
Obviously coproducts of representables are projective and 
enough projectives exist because  
\[ \coprod_{X, x \in G(X)} h_X \rightarrow G \] is an epimorphism.  
Projective objects are thus the coproducts of retracts of representables. That representables are closed under retract themselves, is precisely the condition that idempotents split.  
Note that in any case a retract of a morphism of the form $F \rightarrow F \amalg \coprod X_i$ is of the same form with retracts of the $X_i$ because $\mathcal{SET}^{I^{\op}}$ is extensive. 
\end{proof}

This structure is actually left generated. $\mathcal{L}_{\proj,\eff}$ is generated by the set $\emptyset \rightarrow h_S$ for $S \in I$ and $\mathcal{SET}^{I^{\op}}$ is cocomplete. 

\begin{PAR}[extension to diagrams] 
If $(\mathcal{L}, \mathcal{R})$ is a weak factorization system in $\mathcal{S}$ and $I$ a small category then there are two natural candidates for a weak factorization system in $\mathcal{S}^I$ namely
$(\mathcal{L}^i, \mathcal{R}^i)$ in which a morphism is in $\mathcal{L}^i$ precisely if it is point-wise in $\mathcal{L}$, and $(\mathcal{L}^p, \mathcal{R}^p)$ in which a morphism is in $\mathcal{R}^p$ precisely if it is point-wise in $\mathcal{R}$. The other class is in each case defined by the lifting property. 
We have the following general result:
\end{PAR}
\begin{PROP}\label{PROPWFSDIAGRAM}
Let $\mathcal{S}$ be a category with weak factorization system and let $I$ be a small category,
If $(\mathcal{L}, \mathcal{R})$ is of projective type and either $I$ is finite, or $\mathcal{S}$ has all coproducts, then $(\mathcal{L}^p, \mathcal{R}^p)$ is a weak factorization system.
If $(\mathcal{L}, \mathcal{R})$ is of injective type and either $I$ is finite, or $\mathcal{S}$ has all products, then $(\mathcal{L}^i, \mathcal{R}^i)$ is a weak factorization system.
\end{PROP}
\begin{proof}
Let $(\mathcal{L}, \mathcal{R})$ be of projective type. By assumption $\mathcal{R}$ is saturated and has enough projectives. Then obviously also $\mathcal{R}^p$ is saturated and 
the class of objects of the form $i_! P$ for $P$ projective w.r.t.\@ $\mathcal{R}$ and $i \in I$  constitute enough projectives for $\mathcal{R}^p$.
In fact, for every object $X \in \mathcal{S}^{I}$ we have a canonical morphism
\[ \coprod_{i \in I} i_! P_i \rightarrow X \]
where $P_i \rightarrow i^*X$ is a projective cover. 
Since $\mathcal{R}$ is right-saturated, it is clear that this morphism is point-wise  in $\mathcal{R}$ and thus in $\mathcal{R}^p$. The other assertion is dual. 
\end{proof}
\begin{BEISPIEL}
The canonical weak factorization system $(\mathcal{L}, \mathcal{R})$ on $\mathcal{SET}$ with \[\mathcal{L}=\text{injections} \quad \mathcal{R}=\text{surjections} \] is left- and right-saturated and 
has enough injectives and projectives. In fact every set is injective and projective. 
We therefore get that $(\mathcal{L}^i, \mathcal{R}^i) = (\mathcal{L}_{inj}, \mathcal{R}_{inj})$ and $(\mathcal{L}^p, \mathcal{R}^p) = (\mathcal{L}_{proj}, \mathcal{R}_{proj})$
are weak factorization systems on $\mathcal{SET}^{I}$. 
\end{BEISPIEL}

\begin{BEM}
It is well-known that if $I$ is directed and $\mathcal{S}$ has colimits then $(\mathcal{L}^p, \mathcal{R}^p)$ is always a weak factorization system and if $I$ is inverse and $\mathcal{S}$ has limits $(\mathcal{L}^i, \mathcal{R}^i)$ is always a weak factorization system. In these cases the factorizations can be constructed explicitly by induction.
\end{BEM}

$(\mathcal{L}_{\inj}, \mathcal{R}_{\inj})$ is a weak factorization system more generally for any topos: 
\begin{PROP}
Let $\mathcal{S}$ be a topos. Then $(\mathcal{L}_{\inj}, \mathcal{R}_{\inj}) = (\mathcal{L}_{\inj, \eff}, \mathcal{R}_{\inj, \eff})$ 
is a left generated weak factorization system. 
\end{PROP}
\begin{proof}
Cf.\@ \cite[Corollaire 1.30]{Cis02} and also \cite[2.1.11]{Cis19}.
\end{proof}

\begin{PAR}[Abelian weak factorization systems (cotorsion pairs)]
If $\mathcal{A}$ is an Abelian category, call a weak factorization system as above {\bf Abelian}, if the morphism $f: A \rightarrow B$ is in $\mathcal{L}$ precisely if $f$ is a monomorphism and $0 \rightarrow \coker(f)$ is in $\mathcal{L}$  (i.e.\@ $\coker(f)$ is projective for the w.f.s.\@) and dually for $\mathcal{R}$. Then it is easy to see that an Abelian weak factorization system is the same as a complete cotorsion pair in the sense of Hovey, Gillespie, etc. (cf.\@ \cite{Hov07}) and morphisms of weak factorization systems as defined in (\ref{PARWFSFUNCTORIALITY}) correspond to compatible adjunctions in several variables as defined by Recktenwald (cf.\@ \cite[Proposition 2.3.21]{Rec19}).

In the extreme case, if $\mathcal{A}$ has enough injectives, we have thus the w.f.s.\@ ($\mathcal{L}_{\inj}$, $\mathcal{R}_{\inj}$) in which $\mathcal{L}_{\inj}$ consists of monomorphisms and  $\mathcal{R}_{\inj}$ consists of epimorphisms with injective kernel, or dually if $\mathcal{A}$ has enough projectives, we have the weak factorization system ($\mathcal{L}_{\proj}$, $\mathcal{R}_{\proj}$)  in which $\mathcal{R}_{\proj}$ consists of epimorphisms and  $\mathcal{L}_{\proj}$ consists of monomorphsims with projective cokernel.  
\end{PAR}

\section{Transport of weak factorization systems to simplicial objects}

\begin{PAR}
Recall the standard model structure on simplicial sets $\mathcal{SET}^{\Delta^{\op}}$ in which
\[
\begin{array}{rcl}
\Cof & =& \text{monomorphisms } \\
\Fib & = & \text{Kan fibrations } \\
\mathcal{W} & = & \text{weak equivalences of simplicial sets } 
\end{array}
\]
As in any model category this gives rise to two weak factorization systems ($\Cof, \Tfib$) and ($\Tcof, \Fib$). Those are left generated by the sets 
 $\{\partial \Delta_{n} \rightarrow \Delta_n\}$, and $\{\Lambda_{n,k} \rightarrow \Delta_n\}$, respectively. 
\end{PAR}
 
 \begin{PAR}
 Alternatively $\Tcof$ can be generated by the set $\{ \Delta_n \times \{e\} \cup  \partial \Delta_n \times \Delta_1 \rightarrow \Delta_n \times \Delta_1 \}$ defined by the Cartesian squares of simplicial sets:
\[
\xymatrix{ \partial \Delta_n \times \{e\} \ar[r]^{\id \times \delta_{1}^{1-e}} \ar[d] & \partial \Delta_n \times \Delta_1 \ar[d] \\ 
 \Delta_n \times \{e\}  \ar[r]^{\id \times \delta_{1}^{1-e}} &  \Delta_n \times \Delta_1
  } \]
 cf.\@ \cite[Chapter IV]{GZ67}.
\end{PAR}

\begin{PAR}\label{PARSIMPLICIALSTRUCT}
Consider a category $\mathcal{S}$ with weak factorization system $(\mathcal{L}, \mathcal{R})$ such that $\mathcal{S}$ has coproducts  (mutatis mutandis one may restrict everywhere to countable coproducts) and finite limits and assume that 
\begin{enumerate}
\item pushouts along morphisms in $\mathcal{L}$ exist;
\item transfinite compositions of morphisms in $\mathcal{L}$ exist. 
\end{enumerate}
Note that the resulting morphisms are automatically in  $\mathcal{L}$ again. 
Then there is a functor in two variables
\[ \otimes: \mathcal{S} \times \mathcal{SET}^{\Delta^{\op}} \rightarrow  \mathcal{S}^{\Delta^{\op}} \]
which is just the tensoring over $\mathcal{SET}$ (which $\mathcal{S}$ has due to the existence of coproducts)  extended to diagrams.
Under the assumptions this functor has right adjoints in both variables (the first restricted to finite simplicial sets):
\[
\begin{array}{lclcll}
\Hom_l:& (\mathcal{SET}^{\Delta^{\op}}_{fin})^{\op} &\times& \mathcal{S}^{\Delta^{\op}} &\rightarrow& \mathcal{S}  \\
\Hom_r:& \mathcal{S}^{\op} &\times& \mathcal{S}^{\Delta^{\op}}  &\rightarrow&  \mathcal{SET}^{\Delta^{\op}} 
\end{array}
\]
Observe that $\Hom_l$ involves an end construction and thus exists because $\mathcal{S}$ has finite limits. More precisely, we have
\[ \Hom_l(X, Y) = \int_{\Delta^{\op}} \Hom(X_{-}, Y_{-}). \]
If $X$ is finite then, in particular, $X_{n'}$ is finite for all $n'$ and thus $\Hom(X_{n'}, Y_n)$ exists, being a finite product. 
Furthermore, there exists an $N \in \N$ such that $X \cong \iota_{N,!} X_{\le N}$ for $\iota_N: \Delta_{\le N}^{\op} \hookrightarrow \Delta^{\op}$. Therefore by Lemma~\ref{LEMMAEND} we have
\[ \Hom_l(X, Y) = \int_{\Delta^{\op}} \Hom((\iota_{N,!} (X_{\le N}))_{-}, Y_{-}) \cong  \int_{\Delta^{\op}_{\le N}}\Hom(X_{-},Y_{-}) \]
and thus the occurring end is a finite limit. 
\end{PAR}

\begin{PAR}\label{PARTRANSPORTWFS}
We want to investigate under which circumstances we get a model category structure on $\mathcal{S}^{\Delta^{op}}$ such that $\otimes$ is a morphism of weak factorization systems in the sense of (\ref{PARWFSFUNCTORIALITY}) with the same finiteness restriction as in (\ref{PARSIMPLICIALSTRUCT}):
\begin{gather*}
  (\mathcal{L}, \mathcal{R}), (\Cof, \Fib \cap \mathcal{W}) \rightarrow (\Cof_{\mathcal{S}}, \Fib_{\mathcal{S}} \cap \mathcal{W}_{\mathcal{S}}) \\
  (\mathcal{L}, \mathcal{R}), (\Cof \cap \mathcal{W}, \Fib) \rightarrow (\Cof_{\mathcal{S}} \cap \mathcal{W}_{\mathcal{S}}, \Fib_{\mathcal{S}}) 
\end{gather*}
\end{PAR}
One can try to transport the weak factorization systems via $\otimes$ in the sense of (\ref{PARWFSTRANSPORT}), that is, define  
\[ \Cof_{\mathcal{S}} := {}^\Box(\{(X \rightarrow Y) \boxplus (\partial \Delta_n \rightarrow \Delta_n) \}^\Box) \quad \text{and} \quad 
\Tcof_{\mathcal{S}} := {}^\Box( \{(X \rightarrow Y) \boxplus (\Lambda_{k, n} \rightarrow \Delta_n)\}^\Box)   \]
where $X \rightarrow Y$ runs through $\mathcal{L}$ (or a generating class therein). 
These are weak factorization systems under pretty general circumstances. 
The first case is automatic: 
\begin{PROP}\label{PROPCOFTFIB}
The pair
$(\Cof_{\mathcal{S}}, \Tfib_{\mathcal{S}})$ (transport of $(\mathcal{L}, \mathcal{R}), (\Cof, \Fib\cap \mathcal{W})$ along $\otimes$) is a weak factorization system. 
If $\mathcal{S}$ has finite colimits then a morphism $X \rightarrow Y$ is in $\Cof_{\mathcal{S}}$ if and only if the morphism 
\[ L_n Y \amalg_{L_n X} X_n \rightarrow Y_n \]
is in $\mathcal{L}$. If $\mathcal{S}$ has finite limits (which we always assume), a morphism $X \rightarrow Y$ is in $\Tfib_{\mathcal{S}}$ if and only if the morphism  is 
\[ X_n \rightarrow M_n X \times_{M_n Y} Y_n  \]
is in $\mathcal{R}$. Hence, if $\mathcal{S}$ has also all finite colimits, then this 
weak faktorization is equivalently the Reedy-structure. 

Furthermore, any morphism in $\Cof_{\mathcal{S}}$ is degree-wise in $\mathcal{L}$ (but this is not always sufficient). 
\end{PROP}
\begin{proof}(Cf. \cite[II, \S 4, Proposition 3]{Qui67})
Given a morphism $X \rightarrow Y$, 
construct a sequence of objects $X:=X^{(0)} \rightarrow X^{(1)} \rightarrow X^{(2)} \rightarrow \cdots$ by means of coCartesian squares of the form 
\[ \xymatrix{ \partial \Delta_n \otimes X^{(n)}_n \ar[r] \ar[d] & \Delta_n \otimes X^{(n)}_n  \ar[d] \\
\partial \Delta_n \otimes P^{(n)} \ar[r] & \lefthalfcap  \ar[r] \ar[d] & X^{(n)} \ar[d] \\
& \Delta_n \otimes P^{(n)} \ar[r]  & X^{(n+1)} 
 }\]Here $P^{(n)}$ is obtained by factoring
\[ \xymatrix{ X^{(n)}_n \ar[rr]^-{\in \mathcal{L}} & & P^{(n)} \ar[rr]^-{\in \mathcal{R}} && Y_n \times_{\Hom(\partial \Delta_n, Y)}  \Hom(\partial \Delta_n, X^{(n)})  } \]
Note that the push-out is degree-wise along a morphism in $\mathcal{L}$ and thus exists by assumption. 
The objects $X^{(n)}$ come equipped with an evident morphism to $Y$. Note that $\colim X^{(n)}$ is degree-wise a transfinite composition of morphisms in $\mathcal{L}$ and thus exists. We claim that
\[ \xymatrix{ X \ar[r]  & \colim X^{(n)} \ar[r] & Y  } \]
is the requested factorization. By construction $X \rightarrow \colim X^{(n)}$ is in $\Cof_{\mathcal{S}}$. 
Since the morphism $X^{(n)}  \rightarrow X^{(n+1)}$  is an isomorphism in degree $< n$ 
we have to show that for a morphism $A \rightarrow B$ in $\mathcal{L}$ a lifting as indicated in the diagram
\[ \xymatrix{ \partial \Delta_n \otimes A \ar[r] \ar[d] & \Delta_n \otimes A  \ar[d] \\
\partial \Delta_n \otimes B \ar[r] & \lefthalfcap  \ar[r] \ar[d] & X^{(n+1)} \ar[d] \\
& \Delta_n \otimes B \ar[r] \ar@{.>}[ru] & Y 
 }\]
 exists --- or equivalently a lift in 
\[ \xymatrix{ A \ar[r] \ar[d] & X_n^{(n+1)} = P^{(n)}  \ar[d]^{\in \mathcal{R}} \\
B \ar[r] \ar@{.>}[ru] & \Box  \ar[rr] \ar[d] & &  Y_n \ar[d] \\
& \Hom(\partial \Delta_n, X^{(n+1)}) = \Hom(\partial \Delta_n, X^{(n)})   \ar[rr] &  & \Hom(\partial \Delta_n, Y)
 }\]
 By construction the so indicated vertical morphism in the diagram is in $\mathcal{R}$ and thus a lift exists.

If $\mathcal{S}$ has finite colimits one can alternatively construct a factorization
\[ X \rightarrow Z \rightarrow Y \]
using the Reedy structure, factorizing inductively the morphism 
\[ L_n Z \amalg_{L_n X} X_n  \rightarrow M_n Z \times_{M_n Y} Y_n \]
into $\mathcal{L}$ and $\mathcal{R}$ for all $n$. 
\end{proof}

The existence of the latching object $L_n$ and the cofibrancy conditions can sometimes be made much more concrete. 
This is due to the fact that $\Delta^{\op}$ is an {\em elegant} Reedy category in the sense of \cite{BR13}. For this to work the class $\mathcal{L}$ of the weak factorization system has to satisfy the following: 

\begin{DEF}\label{DEFADHESIVE}
A weak factorization system $(\mathcal{L}, \mathcal{R})$ is called {\bf adhesive} if the following holds: 
\begin{enumerate} \item[(A1)] All morphisms in $\mathcal{L}$ are monomorphisms.
\item[(A2)] For any {\em Cartesian} square
\[ \xymatrix{
A \cap B \ar[r] \ar[d] & B \ar[d] \\
A \ar[r] & X  }\]
in which all morphisms are in $\mathcal{L}$, the push-out exists and the morphism
\[ A \cup B := A \amalg_{A \cap B} B \rightarrow X \]
is in $\mathcal{L}$. 
\item[(A3)] Push-outs $A \cup B$ as in 2.\@ are stable under intersection. More precisely, for morphisms $A \rightarrow X$, $B \rightarrow X$, and $C \rightarrow X$
in $\mathcal{L}$, such that all (multiple) intersections exist and all ``projections'' are in $\mathcal{L}$, the diagram
\[ \xymatrix{
A \cap B \cap C  \ar[r] \ar[d] & A \cap C \ar[d] \\
B \cap C \ar[r] & (A \cup B) \cap C  }\]
is a push-out.
\end{enumerate}
\end{DEF}

We say that a subobject is an {\bf $\mathcal{L}$-subobject} if one (hence all) representing monomorphism is in $\mathcal{L}$. Consider a poset $\mathcal{X}$ of $\mathcal{L}$-subobjects of $X$. We assume that also all morphisms in the poset are in $\mathcal{L}$. We say that $\mathcal{X}$ is {\bf closed under intersection}, if for two objects $A \rightarrow X$ and $B \rightarrow X$ also $A \cap B \rightarrow X$ is in $\mathcal{X}$ and also the projections are morphisms in $\mathcal{X}$ (in particular, they are in $\mathcal{L}$).
We say that $\mathcal{X}$ is {\bf closed under unions}, if for two objects $A \rightarrow X$ and $B \rightarrow X$ also $A \cup B \rightarrow X$ is an object in $\mathcal{X}$ and the morphisms $A \rightarrow A \cup B$ and $B \rightarrow A \cup B$ are morphisms in $\mathcal{X}$.

The following follows directly from the definition: 
\begin{LEMMA}
Let $(\mathcal{L}, \mathcal{R})$ be an adhesive weak factorization system.
Let $X$ be an object and consider a poset $\mathcal{X}$ of $\mathcal{L}$-subobjects of $X$ with the assumption before, closed under intersections and unions. Then $\mathcal{X}$ is a distributive lattice. If $\mathcal{X}$ is only closed under intersection, then we may add all unions and obtain a larger poset $\widetilde{\mathcal{X}}$ which is still closed under intersection. 
\end{LEMMA}

\begin{LEMMA}
\begin{enumerate}
\item If $\mathcal{S}$ is adhesive (e.g.\@ a topos) then $(\mathcal{L}_{\inj}, \mathcal{R}_{\inj})$ is adhesive.
\item If $\mathcal{S}$ is extensive then $(\mathcal{L}_{\proj,\spl}, \mathcal{R}_{\proj,\spl})$ is adhesive. 
\item If $\mathcal{S}$ is extensive, has pull-backs and enough projectives, then $(\mathcal{L}_{\proj,\eff}, \mathcal{R}_{\proj,\eff})$ is adhesive (for example $\mathcal{S} = \mathcal{SET}^I$).
\end{enumerate}
\end{LEMMA}
\begin{proof}1.\@ For an adhesive category and $\mathcal{L}_{\inj}$ only (A2) does not follow immediately from the definition. For a proof see \cite[Proposition 2.4]{GL12}.
2.\@ For an extensive category $\mathcal{L}_{\proj,\spl}$ is precisely the class of coproduct injections and also only (A2) does not follow immediately from the definition. It is left as exercise. 
3.\@ Under the assumptions on  $\mathcal{S}$ the class $\mathcal{L}_{\proj,\eff}$ is also a class of coproduct injections and the same proof as in 2.\@ applies. 
\end{proof}

\begin{PROP}\label{PROPADHESIVE}
Let $\mathcal{S}$ be a category and $(\mathcal{L}, \mathcal{R})$ be an adhesive weak factorization system.   
\begin{enumerate}
\item An object $X$ in $\mathcal{S}^{\Delta^{\op}}$ is cofibrant if and only if all $X_n$ are projective (w.r.t.\@ $(\mathcal{L}, \mathcal{R})$) and
 all degeneracies $X_n \rightarrow X_m$ (which are automatically split monomorphisms) are in $\mathcal{L}$. Furthermore 
$L_n X$ exists for any cofibrant $X$ and it is just the join of the degeneracy $\mathcal{L}$-subobjects of $X_n$. 
\item If $\mathcal{S}$ is extensive, and if $\mathcal{L}$ consists of morphisms $A \rightarrow A \amalg B$ with $B$ projective, then 
$X$ is cofibrant if and only if $X_n = X_{n,deg} \amalg X_{n,nd}$ with $X_{n,deg} = \coprod_{\Delta_n \twoheadrightarrow \Delta_k, n \not=k} X_{k,nd}$ and all $X_{n,nd}$ are projective. 
\end{enumerate}
\end{PROP}
\begin{proof}
Let $X$ be a simplicial object in which all degeneracies are in  $\mathcal{L}$. Let $L_n \subset \Delta^{\op,-} \times_{/\Delta^{\op,-}} \Delta_m$ be the latching diagram. 
Using the fact that $\Delta^{\op}$ is elegant, for every two morphisms $\alpha: \Delta_{k_1} \rightarrow \Delta_m$ and $\beta: \Delta_{k_2} \rightarrow \Delta_m$ there is an {\em absolute} pull-back $\alpha \cap \beta: \Delta_{k'} \rightarrow \Delta_m$ in $\Delta^{\op}$ which is again a degeneracy. It follows that the colimit over $L_n$ (in $\mathcal{S}$) is the union (in the sense of $\mathcal{L}$-subobjects) of the components. Therefore $L_n X$ exists and $L_n X \rightarrow X_n$ is in $\mathcal{L}$ and thus $X$ is cofibrant. (For $n=0$ the union $L_n X$ is empty and the assumption that $X_0$ is projective is needed). For the converse factor the morphism $\emptyset \rightarrow X$ as  $\emptyset \rightarrow X' \rightarrow X$ as in the proof of Proposition~\ref{PROPCOFTFIB}. By inspection every degeneracy in $X'$ is in $\mathcal{L}$. However, since $X' \rightarrow X$ is a trivial fibration, $X$ is a retract of $X'$ and therefore also $X$ has this property. 

For 2.\@ note that a Cartesian and coCartesian square in which all morphisms are in $\mathcal{L}$ has the form
\[ \xymatrix{
A \ar[r] \ar[d] & A \amalg B \ar[d] \\
A \amalg C \ar[r] & A \amalg B \amalg C  }\]
We leave the details as an exercise. 
\end{proof}

We come back to the task of transporting weak factorization systems along $\otimes$. 
The second transport of structure in \ref{PARTRANSPORTWFS} exists under some mild hypotheses:

\begin{PROP}\label{PROPTCOFFIB}
The pair $(\Tcof_{\mathcal{S}}, \Fib_{\mathcal{S}})$ (transport of $(\mathcal{L}, \mathcal{R}), (\Cof\cap \mathcal{W}, \Fib)$ via $\otimes$) on $\mathcal{S}^{\Delta^{\op}}$ is a weak factorization system if one of the following holds true
\begin{enumerate}
\item $\mathcal{L}$ is generated by a class [sic.] with $\kappa$-small domain and codomain for some cardinal $\kappa$.
\item $\mathcal{L}$ is generated by a class of morphisms of the form $\emptyset \rightarrow P$  and every object in $\mathcal{S}^{\Delta^{\op}}$ is fibrant, i.e.\@ the morphism $X \rightarrow \cdot$ is in $\Fib_{\mathcal{S}}$. This follows e.g.\@ if for every object $X$ there exists 
 a factorization $\emptyset \rightarrow \widetilde{X} \rightarrow X$ in the weak factorization system such that $\widetilde{X}$ is a cogroup object. 
\end{enumerate}
\end{PROP}
\begin{proof}
1.\@ This is a slightly modified small object argument. Choose a well-ordering on $\kappa$. 
Set $X^{(0)}:=X$ and for all $\alpha \in \kappa$ and $0 \le k \le n$ define $P^{(\alpha,n,k)}$ by factoring
\[ \xymatrix{ X^{(\alpha)}_n \ar[rr]^-{\in \mathcal{L}} & & P^{(\alpha,n,k)} \ar[rr]^-{\in \mathcal{R}} && Y_n \times_{\Hom(\Lambda_{n,k}, Y)}  \Hom(\Lambda_{n,k}, X^{(\alpha)}).  } \]
Then construct the push-out
\[ \xymatrix{ \coprod_{0 \le k\le n} \Lambda_{n,k} \otimes P^{(\alpha,n,k)} \amalg_{\Lambda_{n,k} \otimes X^{(\alpha)}_n} \Delta_n \otimes X^{(\alpha)}_n  \ar[r] \ar[d] &  X^{(\alpha)} \ar[d]  \\
\coprod_{0 \le k\le n} \Delta_n \otimes P^{(\alpha,n,k)}_n \ar[r] & X^{(\alpha+1)} }
\]
By transfinite induction we may continue and construct a $\kappa$-sequence and get a factorization
\[ \xymatrix{ X \ar[r]  & \colim X^{(\alpha)} \ar[r] & Y  } \]
Note that $X^{(\alpha)}$ is degree-wise a sequence of morphisms in $\mathcal{L}$ and thus the colimit exists by assumption. 
By construction $X \rightarrow \colim X^{(\alpha)}$ is in $\Cof_{\mathcal{S}}$. 

We have to prove that $\colim X^{(\alpha)} \rightarrow Y$ is in $\Tfib_{\mathcal{S}}$. 
It suffices that for all morphisms $A \rightarrow B$ in the generating class of $\mathcal{L}$, and all $0 \le k \le n$, there exists a lift as indicated in the following diagram:
\[ \xymatrix{ \Lambda_{n,k} \otimes A \ar[r] \ar[d] & \Delta_n \otimes A  \ar[d] \\
\Lambda_{n,k} \otimes B \ar[r] & \lefthalfcap  \ar[r] \ar[dd] & X^{(\alpha)} \ar[d] \\
&&  \ar[d] X^{(\alpha+1)}  \\
& \Delta_n \otimes B \ar[r] \ar@{.>}[ru] & Y 
 }\]
Note that, by assumption on $\kappa$-smallness, the morphism $\lefthalfcap \rightarrow \colim X^{(\alpha)}$ factors through one of the $X^{(\alpha)}$.
 
 This can be constructed choosing a lift in 
\[ \xymatrix{ A \ar[r] \ar[dd] & X_n^{(\alpha)} \ar[d] \ar[r] & X_n^{(\alpha+1)} \ar[dd] \\
&  P^{(\alpha,n,k)}_n \ar[d]^{\in \mathcal{R}} \ar[ru] &  \\
B \ar[r] \ar@{.>}[ru] & \Box  \ar[r] \ar[d] & \Box \ar[r]  \ar[d] &   Y_n \ar[d] \\
& \Hom( \Lambda_{n,k}, X^{(\alpha)})  \ar[r] &  \Hom(\Lambda_{n,k}, X^{(\alpha+1)})   \ar[r]   & \Hom(\Lambda_{n,k}, Y)
 }\]
 By construction the indicated morphism is in $\mathcal{R}$ and thus a lift exists.
 
2.\@ will be shown in the proof of Theorem~\ref{THEOREMMODELSOBJECTSI}, 2.\@ below.
\end{proof}

\begin{LEMMA}\label{LEMMATRIVIAL}
We always have $\Tcof_{\mathcal{S}} \subset \Cof_{\mathcal{S}}$ and $\Tfib_{\mathcal{S}} \subset \Fib_{\mathcal{S}}$.
\end{LEMMA}
\begin{proof}
The statements are clearly equivalent hence it suffices to see $\Tfib_{\mathcal{S}} \subset \Fib_{\mathcal{S}}$.
The class $\Fib$ resp.\@ $\Tfib$, may also be described as the class of morphisms $f$ such that
\[ \boxdot\Hom(g, f) \in \Fib  \quad (\text{resp.\@ }\in \Fib \cap \mathcal{W})  \]
for all $g \in \mathcal{L}$ using (\ref{PARWFSTRANSPORT}). Since $\Fib \cap \mathcal{W} \subset \Fib$ the statement follows. 
\end{proof}

We have the following proposition clarifying the functoriality of the construction:

\begin{PROP}\label{PROPFUNCTOR}
Let $\mathcal{S}_0, \mathcal{S}_1$ be categories with weak factorization systems $(\mathcal{L}_0, \mathcal{R}_0), (\mathcal{L}_1, \mathcal{R}_1)$ and
let $F: \mathcal{S}_0 \rightarrow \mathcal{S}_1$ be a functor.
Then:
\begin{enumerate}
\item If $F$ commutes with finite limits and $F(\mathcal{R}_0) \subset F(\mathcal{R}_1)$ then $F: \mathcal{S}_0^{\Delta^{\op}} \rightarrow  \mathcal{S}_1^{\Delta^{\op}}$ satisfies
\[ F(\Fib_{\mathcal{S}_0}) \subseteq \Fib_{\mathcal{S}_1}
\quad \text{and} \quad 
F(\Tfib_{\mathcal{S}_0}) \subseteq \Tfib_{\mathcal{S}_1} \]
\item If $F$ commutes with coproducts, push-outs along morphisms in $\mathcal{L}$, transfinite compositions of morphisms in $\mathcal{L}$, and $F(\mathcal{L}_0) \subset F(\mathcal{L}_1)$, and
Theorem~\ref{PROPCOFTFIB} applies for $(\mathcal{L}_0, \mathcal{R}_0)$ then
\[ F(\Cof_{\mathcal{S}_0}) \subseteq \Cof_{\mathcal{S}_1}
\quad \text{and} \quad 
F(\Tcof_{\mathcal{S}_0}) \subseteq \Tcof_{\mathcal{S}_1}. \]
\end{enumerate}
\end{PROP}
\begin{proof}
1.\@  $\Fib_{\mathcal{S}_0}$ can be characterized as those morphisms $f$ such that
\[ \Hom(\Lambda_{k,n} \hookrightarrow \Delta_n, f) \in \mathcal{L}_0 \]
for all $0 \le k \le n$. If $F$ commutes with finite limits then it follows that
\[ \Hom(\Lambda_{k,n} \hookrightarrow \Delta_n, F(f)) \in F(\mathcal{L}_0) \subset \mathcal{L}_1. \]
The same argument applies for $\Tfib$. 

2.\@ Since Theorem~\ref{PROPCOFTFIB} applies, by the retract argument, it suffices to see that all cofibrations constructed in the proof of the Theorem are mapped to $\Tcof_{\mathcal{S}_1}$. 
Using the assumptions on $F$ it suffices to see  
\[ F((\Lambda_{k,n} \hookrightarrow \Delta_n) \boxplus f) \subset \Tcof_{\mathcal{S}_1} \]
for $f \in \mathcal{L}_0$. This follows because $F(\mathcal{L}_0) \subset \mathcal{L}_1$ and  
\[ F((\Lambda_{k,n} \hookrightarrow \Delta_n) \boxplus f)  \cong (\Lambda_{k,n} \hookrightarrow \Delta_n) \boxplus F(f).  \]
because $F$ commutes with coproducts.
\end{proof}

\section{Model category structures on simplicial objects}

Assume that also $(\Tcof_{\mathcal{S}}, \Fib_{\mathcal{S}})$ is a weak factorization system (see Proposition~\ref{PROPTCOFFIB}) and set 
\[ \mathcal{W}_{\mathcal{S}}:=\Tcof_{\mathcal{S}} \cdot \Tfib_{\mathcal{S}}. \]
We investigate under which circumstances  the resulting triple $(\Cof_{\mathcal{S}}, \Fib_{\mathcal{S}}, \mathcal{W}_{\mathcal{S}})$ is a simplicial model category. 
Note that only the 2-out-of-3 property of $\mathcal{W}$ is in question\footnote{Using lemma \ref{LEMMATRIVIAL}, the retract argument shows $\Tfib = \Fib \cap \mathcal{W}$ and $\Tcof = \Cof \cap \mathcal{W}$.}. 
There is no way of showing this property directly. 

In many important cases the class $\Tcof_{\mathcal{S}}$ must be enlarged (and thus $\Fib_{\mathcal{S}}$ reduced)  for $\mathcal{W}_{\mathcal{S}}$ to satisfy 2-out-of-3. Cisinski \cite{Cis02} investigates this more generally for topoi with a given interval object. His work can be applied to our situation in the case that $\mathcal{S}$ is a topos equipped 
with the pair $(\mathcal{L}_{\inj}, \mathcal{R}_{\inj})$. In the sequel, however, we will investigate only situations where the four classes obtained via transport along $\otimes$ assemble to a model category structure on the nose, but allow more general $\mathcal{S}$ and weak factorization systems. This slightly generalizes  Quillen's original theory of model category structures on categories of simplicial objects \cite[II.4]{Qui67}.

\begin{PAR}\label{PARCOND}
We exhibit a class $\mathcal{P}$ of projective objects $X \in \mathcal{S}$ (i.e.\@ those such that $\emptyset \rightarrow X$ is in $\mathcal{L}$) and define
\[ \mathcal{W}_{\mathcal{S}} := \{ f \ | \ \Hom(X, f) \in \mathcal{W}\  \forall X \in \mathcal{P}\}. \]
We call the morphisms in $\mathcal{W}_{\mathcal{S}}$ {\bf weak equivalences}. 
The 2-out-of-3 property is now clear and we have to show the following properties:
\begin{enumerate}
\item $\Tcof_{\mathcal{S}} \subset \mathcal{W}_{\mathcal{S}}$,
\item $\Tfib_{\mathcal{S}} \subset \mathcal{W}_{\mathcal{S}}$,
\item $\Fib_{\mathcal{S}} \cap \mathcal{W}_{\mathcal{S}} \subset \Tfib_{\mathcal{S}}$.
\end{enumerate}
(The fourth property $\Cof_{\mathcal{S}} \cap \mathcal{W}_{\mathcal{S}} \subset \Tcof_{\mathcal{S}}$ then follows  from the retract argument.)
\end{PAR}

\begin{LEMMA}\label{LEMMATFIB}
\begin{enumerate}
\item  $\Tfib_{\mathcal{S}} \subset \mathcal{W}_{\mathcal{S}}$ (Property 2.\@ above) holds true.

 \item If $\mathcal{L}$ is generated by a class of morphisms $X \rightarrow Y$ with $X$ and $Y$ in $\mathcal{P}$ then 
 $\Fib_{\mathcal{S}} \cap \mathcal{W}_{\mathcal{S}} \subset \Tfib_{\mathcal{S}}$  (Property 3.\@ above) holds true. 
 \item If $\mathcal{L}$ is even generated by $\{ \emptyset \rightarrow X \ |\ X \in \mathcal{P}\}$ then $f \in \Fib_{\mathcal{S}}$ (resp. $\in \Tfib_{\mathcal{S}}$) if and only if
 $\Hom(X, f) \in \Fib$ (resp. $\in \Fib \cap \mathcal{W}$) for all $X \in \mathcal{P}$. 
\end{enumerate}

\end{LEMMA}
\begin{proof}
1. Let $f \in \Tfib_{\mathcal{S}}$ and let $X \in \mathcal{P}$. Since the morphism $\partial \Delta_n \otimes X \rightarrow \Delta_n \otimes X$ is in $\Cof$ it has the left lifting property w.r.t.\@ $f$. 
Using the adjunction we get that $\partial \Delta_n \rightarrow \Delta_n$ has the left lifting property w.r.t.\@ $\Hom(X, f)$. The latter is thus a trivial Kan fibration of simplicial sets, in particular, a weak equivalence. 

2. We can equivalently describe $\Tfib_{\mathcal{S}}$ and $\Fib_{\mathcal{S}}$ as the class of morphisms $g$ such that $\boxdot\Hom(f, g) \in \Fib \cap \mathcal{W}$, resp.\@ $\in \Fib$, for all morphisms $f:A \rightarrow B$ in $\mathcal{L}$.
Here it suffices to restrict to a generating class of $\mathcal{L}$.
We claim that if $f \in \mathcal{L}$ is  a morphism between projective objects in $\mathcal{P}$ then if $g: X \rightarrow Y \in \Fib_{\mathcal{S}} \cap \mathcal{W}_{\mathcal{S}}$ then also $\boxdot\Hom(f, g) \in \mathcal{W}$. From this the assertion follows immediately.
To prove the claim consider the diagram
 of simplicial sets
\[ \xymatrix{
 \Hom(B, X) \ar[rd] \ar[rrrd]^{\in \Fib \cap \mathcal{W}} \ar[rdd] \\
& \Box \ar[rr]_-{\in \Fib \cap \mathcal{W}} \ar[d] & & \Hom(B, Y) \ar[d] \\
& \Hom(A, X) \ar[rr]_-{\in \Fib \cap \mathcal{W}} & &  \Hom(A, Y) 
} \]
Hence by 2-out-of-3 we get $\boxdot\Hom(f, g) \in \mathcal{W}$.

3. Clear. 
\end{proof}

\begin{LEMMA}\label{LEMMARIGHTPROPER1}
Pullbacks of weak equivalences along fibrations are weak equivalences. 
\end{LEMMA}
\begin{proof} Let
\[ \xymatrix{
  A \times_C B \ar[r] \ar[d] & A \ar[d]^f \\
 B \ar[r]_g & C
} \]
be a Cartesian diagram of objects with $g$ a fibration and $f$ a weak equivalence. Let $X \in \mathcal{P}$. Then we also have a Cartesian diagram 
\[ \xymatrix{
 \Hom(X,A \times_C B) \ar[r] \ar[d] & \Hom(X,A) \ar[d]^f \\
 \Hom(X, B) \ar[r]_g & \Hom(X, C)
} \]
and $g$ is a fibration (because $X$ is projective) and $f$ is a weak equivalence by assumption. Therefore $\Hom(X,A \times_C B ) \rightarrow  \Hom(X, B)$ is a weak equivalence. 
Since this holds for all $X$, the pullback $A \times_C B  \to B$ is again a weak equivalence. 
\end{proof}

The following is a slight generalization of Quillen's Theorem \cite[II, \S 4, Theorem 1]{Qui67}.

\begin{SATZ}\label{THEOREMMODELSOBJECTSI}
Let $\mathcal{S}$ be a category with 
coproducts and finite limits and
let $(\mathcal{L}, \mathcal{R})$ be a weak factorization system such that pushouts of morphisms in $\mathcal{L}$ exist and
transfinite compositions of morphisms in $\mathcal{L}$ exist. 
Choose a class $\mathcal{P}$ as in \ref{PARCOND}. 

Then $\mathcal{S}^{\Delta^{\op}}$ with the classes $\Cof_{\mathcal{S}}$,
$\Fib_{\mathcal{S}}$, and $\mathcal{W}_{\mathcal{S}}$, defined in \ref{PARTRANSPORTWFS} and \ref{PARCOND}, is a right proper simplicial model category in the following cases:
\begin{enumerate}
\item $\mathcal{P}$ is a class of $\N$-small objects such that $\{\emptyset \rightarrow X \ |\ X \in \mathcal{P}\}$ generates $\mathcal{L}$
and $\mathcal{S}$ has $\N$-filtered colimits.
\item $\mathcal{P}$ is the class of all projective objects, $\{\emptyset \rightarrow X \ |\ X \in \mathcal{P}\}$ generates $\mathcal{L}$, and every object in $\mathcal{S}^{\Delta^{\op}}$ is fibrant.  
\end{enumerate}
Let $I$ be a small category. 
In both cases, denoting by $(\mathcal{L}^p, \mathcal{R}^p)$ the projective extension to $\mathcal{S}^{I}$ (i.e.\@ in which $\mathcal{R}^p$ is the point-wise extension of $\mathcal{R}$, e.g.\@ obtained by Proposition~\ref{PROPWFSDIAGRAM}) then also $\mathcal{S}^{I, \Delta^{\op}}$ with the classes $\Cof_{\mathcal{S}^I}$,
$\Fib_{\mathcal{S}^I}$, and $\mathcal{W}_{\mathcal{S}^I}$ is a simplicial model category and $\mathcal{W}_{\mathcal{S}^I}$ (resp.\@ $\Fib_{\mathcal{S}^I}$) consist of those morphisms that are
point-wise in $\mathcal{W}_{\mathcal{S}}$ (resp.\@ $\Fib_{\mathcal{S}}$).
\end{SATZ}

\begin{BEISPIEL}
\begin{enumerate}
\item Let  $\mathcal{S}$ be a category with finite limits, coproducts, and $\N$-filtered colimits. Consider the weak factorization system $(\mathcal{L}_{\proj,\eff}, \mathcal{R}_{\proj,\eff})$ and assume that for every 
object $X$ there is an effective epimorphism $\prod_{i}P_i \rightarrow  X$ where the $P_i$ are $\N$-small. Then Theorem~\ref{THEOREMMODELSOBJECTSI}, 1.\@ applies\footnote{Indeed, $\emptyset \rightarrow P$ for $P$ projective obviously generate $\mathcal{L}_{\proj,\eff}$. But for each such $P$ there is $\prod_{i}P_i \rightarrow P$ effective epimorphism with $P_i$ $\N$-small, which has a section. Therefore also $\emptyset \rightarrow P$ with $P$ $\N$-small generate $\mathcal{L}_{\proj,\eff}$. }. 
\item Let $I$ be a small category and set 
$\mathcal{S} = \mathcal{SET}^I$ with $(\mathcal{L}_{\proj}, \mathcal{R}_{\proj})$ 
and $\mathcal{P}$ is the set of representable objects. Those are connected and therefore Theorem~\ref{THEOREMMODELSOBJECTSI}, 1.\@ applies. This yields the projective model structure on simplicial presheaves. 
\end{enumerate}
\end{BEISPIEL}

\begin{proof}[Proof of Theorem~\ref{THEOREMMODELSOBJECTSI} (Quillen)] 
1.\@ Lemma~\ref{LEMMATFIB}, 3.\@ applies. 
If $\mathcal{S}$ has $\N$-filtered colimits then there is a functor
\[ \Ex^{\infty}: \mathcal{S}^{\Delta^{\op}} \rightarrow \mathcal{S}^{\Delta^{\op}} \]
satisfying $\Hom(P, \Ex^{\infty} X) = \Ex^{\infty} \Hom(P, X)$ for all $\N$-small $P$. It follows that the map $e_X: X \rightarrow \Ex^{\infty} X$ is in $\mathcal{W}_{\mathcal{S}}$. 
A morphism $f$ is a fibration if and only if $\Hom(P, f)$ is a fibration (Lemma~\ref{LEMMATFIB}, 3.). Hence $\Ex^{\infty} X \rightarrow \cdot$ is a fibration.
So if $f$ is in $\Tcof_{\mathcal{S}}$ then we may lift successively in 
\[ \xymatrix{
A \ar[r] \ar[d]_f & \Ex^\infty(A) \ar[d] \\
B \ar[r] \ar@{.>}[ru]^h & \cdot
}\qquad
\xymatrix{
A \ar[rr]^-{ce_Bf} \ar[d]_f  & & \Hom(\Delta_1, \Ex^\infty(B)) \ar[d] \\
B \ar[rr]_-{(e_B, \Ex^{\infty}(f)h)} \ar@{.>}[rru]  & & \Ex^\infty(B) \times \Ex^\infty(B)
}
\]
Therefore $e_B$ (which is a weak equivalence) is right homotopic to $\Ex^{\infty}(f)h$ and hence this is a weak equivalence. Since also $hf$ is a weak equivalence, $h$, and finally $f$, are weak equivalences. 

2. Lemma~\ref{LEMMATFIB}, 1.\@ and 2.\@ apply. Therefore 
$\Tfib_{\mathcal{S}} = \Fib_{\mathcal{S}} \cap  \mathcal{W}_{\mathcal{S}}$.
Note that we do not yet have factorizations for $(\Tcof_{\mathcal{S}}, \Fib_{\mathcal{S}})$ and will therefore show both inclusions for
$\Tcof_{\mathcal{S}} = \Cof_{\mathcal{S}} \cap \mathcal{W}_{\mathcal{S}}$.

If $f: A\rightarrow B$ is a morphism between fibrant objects then 
\[ \xymatrix{ A \ar[r]^-i & A \times_B \Hom(\Delta_1, B) \ar[r]^-p & B } \]
is a factorization into a strong deformation retract followed by a fibration. Note that the second morphism is a composition of a pullback of $\Hom(\Delta_1, B) \rightarrow B \times B$ with the projection $A\times B \rightarrow B$ and thus it is in $\Fib_{\mathcal{S}}$ by purely formal considerations (the factorization is not needed). 
Then $i$ is a weak equivalence because $\Hom(P, i)$ is a strong deformation retract again. 
If $f$ is in $\Tcof_{\mathcal{S}}$ then $f$ is a retract of $i$ and thus a weak equivalence. 
Conversely, if $f$ is in morphism $\Cof_{\mathcal{S}} \cap \mathcal{W}$ then $p$ is a trivial fibration. Hence $f$ is a retract of $i$ and thus 
 a strong deformation retract as well. Thus is has the left lifting property w.r.t.\@ $\Fib_{\mathcal{S}}$ and is thus in $\Tcof_{\mathcal{S}}$. 

To get the missing factorization factor $i = kj$ with $j$ cofibration and $k$ trivial fibration. Then $j$ is a weak equivalence and thus in $\Tcof_{\mathcal{S}}$.
Thus $f = (pk)j$ is the required factorization. 

In both cases the right properness follows from Lemma~\ref{LEMMARIGHTPROPER1}. The assertion on diagram categories follows because $(\mathcal{L}^p, \mathcal{R}^p)$ and the class
\[ \mathcal{P}^p := \{i_! X \ |\ i \in I, X \in \mathcal{P} \} \]
satisfy the same assumptions.
\end{proof}

\section{The split-projective model category}\label{SECTIONSPLITPROJECTIVE}

With slightly modified techniques we prove the following variant which is the novelty of this article:  

\begin{SATZ}\label{THEOREMMODELSOBJECTSII}
Let $\mathcal{S}$ be an extensive category with finite limits and consider the weak factorization system
$(\mathcal{L}_{\proj, \spl}, \mathcal{R}_{\proj, \spl})$. Let $\mathcal{P}$ be a class of $\N$-small objects and assume that every object of $\mathcal{S}$ is
a coproduct of these. 

Then $\mathcal{S}^{\Delta^{\op}}$ with the classes $\Cof_{\mathcal{S}}$,
$\Fib_{\mathcal{S}}$, and $\mathcal{W}_{\mathcal{S}}$, defined in \ref{PARTRANSPORTWFS} and \ref{PARCOND}, is a right proper simplicial model category.

Furthermore, if $(\mathcal{L}_{\proj, \spl}^p, \mathcal{R}_{\proj, \spl}^p)$ is the weak factorization system on $\mathcal{S}^{I}$ such that $\mathcal{R}_{\proj, \spl}^p$ is the point-wise extension of $\mathcal{R}_{\proj, \spl}$ (i.e.\@ obtained by Proposition~\ref{PROPWFSDIAGRAM}) then also $\mathcal{S}^{I, \Delta^{\op}}$, with the classes $\Cof_{\mathcal{S}^I}$,
$\Fib_{\mathcal{S}^I}$, and $\mathcal{W}_{\mathcal{S}^I}$ is a simplicial model category, and $\mathcal{W}_{\mathcal{S}^I}$ (resp.\@ $\Fib_{\mathcal{S}^I}$) consist of those morphisms that are
point-wise in $\mathcal{W}_{\mathcal{S}}$ (resp.\@ in $\Fib_{\mathcal{S}}$).
\end{SATZ}

Recall that $(\mathcal{L}_{\proj, \spl}, \mathcal{R}_{\proj, \spl})$ is adhesive (cf.\@ Definition~\ref{DEFADHESIVE}) in an extensive category and thus we get a very precise description of the cofibrant objects (cf.\@ Proposition~\ref{PROPADHESIVE}).  

Recall that a connected object $X$ in a category is an object such that $\Hom(X, -)$ commutes with coproducts. The following is therefore a direct consequence of the definitions: 
\begin{LEMMA}\label{LEMMACONNECTED}
Let $\mathcal{S}$ be an extensive category with the weak factorization system $(\mathcal{L}_{\proj, \spl}, \mathcal{R}_{\proj, \spl})$.
If $X$ is connected then the functor $\Hom(X, -)$ commutes with push-outs along morphisms in $\mathcal{L}_{\proj, \spl}$ and with transfinite compositions of morphisms in  $\mathcal{L}_{\proj, \spl}$.
\end{LEMMA}

\begin{proof}[Proof of \ref{THEOREMMODELSOBJECTSII}, special case.]
We first consider the special case in which all objects of $\mathcal{P}$ are connected. This applies, for instance, to $\mathcal{S}$ being the free coproduct completion of a category  and $\mathcal{P}$ the class of all connected objects. Note that $\mathcal{L}$ is generated by the morphisms $\emptyset \rightarrow P$ and thus
Proposition~\ref{PROPTCOFFIB}, 1. applies. We thus get a weak factorization system $(\Tcof_{\mathcal{S}}, \Fib_{\mathcal{S}})$. 
Lemma~\ref{LEMMATFIB}, 3.\@ applies. Hence we are reduced to show $\Tcof_{\mathcal{S}} \subset \mathcal{W}_{\mathcal{S}}$. Furthermore,
 $\Fib$ consists precisely of the morphisms $f$ such that $\Hom(P, f)$ is a fibration for all $P \in \mathcal{P}$. 
By the construction in the proof of Proposition~\ref{PROPCOFTFIB} (keeping in mind $\kappa=\aleph_0$ here) every element in $\Tcof$ is a (retract of an) $\N$-transfinite composition of pushouts
of $f: \Lambda_{n,k} \otimes X \rightarrow \Delta_{n} \otimes X$. The latter are weak equivalences because $\Hom(P, f) = (\Lambda_{n,k} \rightarrow \Delta_n) \otimes \Hom(P,X)$ ($P$ is connected)
which is a (possibly infinite) union of trivial cofibrations of simplicial sets and so a trivial cofibration and hence a weak equivalence. 
Any push-out of $f$ and transfinite compositions of such push-outs is therefore also a weak equivalence by Lemma~\ref{LEMMACONNECTED}.
\end{proof}

The general case will be be reduced to this special case using the coproduct completion. For this we need a couple of Lemmas.

\begin{LEMMA}\label{LEMMAADJ}
Let 
\[ \xymatrix{    \mathcal{S}_0 \ar@<2.5pt>[r]^U & \ar@<2.5pt>[l]^C  \mathcal{S}_1  } \] 
be an adjunction of categories with weak factorization systems $(\mathcal{L}_0, \mathcal{R}_0)$ and $(\mathcal{L}_1, \mathcal{R}_1)$  
in the sense of \ref{PARWFSFUNCTORIALITY}. Assume in addition that 
\begin{enumerate}
\item the right adjoint $U$ is fully faithful,
\item $C$ commutes with finite limits,
\item $C(\mathcal{R}_1) \subset \mathcal{R}_0$.
\end{enumerate}
If $(\Tcof_{\mathcal{S}_1}, \Fib_{\mathcal{S}_1})$ is a weak factorization system using the construction in Proposition~\ref{PROPTCOFFIB}, also $(\Tcof_{\mathcal{S}_0}, \Fib_{\mathcal{S}_0})$ is a weak factorization system.
Furthermore $\Tcof_{\mathcal{S}_0} = rc(C(\Tcof_{\mathcal{S}_1}))$ where $rc$ means ``closure under retracts''.
\end{LEMMA}
\begin{proof}
Denote the induced functors on simplicial objects by the same letters $U$ and $C$. 
By Proposition~\ref{PROPFUNCTOR}, $C$ preserves trivial cofibrations and fibrations. 
Therefore, given a morphism $f$ in $\mathcal{S}_0^{\Delta^{\op}}$, factor $U(f)$ as trivial cofibration followed by fibration. Since $U$ is fully-faithful, applying $C$ to this factorization we get a a factorization of $f$ into trivial cofibration followed by a fibration.  
If $f \in \Tcof_{\mathcal{S}_0}$ the retract argument shows thus that $f$ is a retract of a morphism in $C(\Tcof_{\mathcal{S}_1})$.
\end{proof}

\begin{LEMMA}\label{LEMMASCOPROD}
Let $\mathcal{S}$ be an extensive category and form classes $\mathcal{W}_{\mathcal{S}}$ and $\mathcal{W}_{\mathcal{S}^{\amalg}}$ as in \ref{PARCOND} using 
a class $\mathcal{P}$ of $\N$-small objects, and the class of all connected objects (i.e.\@ the image of $\mathcal{S} \hookrightarrow \mathcal{S}^{\amalg}$), respectively. 
Then
\[ \amalg(\mathcal{W}_{\mathcal{S}^{\amalg}}) \subset \mathcal{W}_{\mathcal{S}}.  \]
\end{LEMMA}
\begin{proof}
 Let $f: X \rightarrow Y$ be a weak equivalence in $\mathcal{S}^{\amalg}$. We have to show that 
 \[ \Hom(P, \amalg f) \]
is a weak equivalence for all $\N$-small $P$. Because we are in an extensive category we have (Proposition~\ref{PROPEXTENSIVE})
 \[ \Hom(P, \amalg f) \cong \colim_{\substack{P \cong \amalg P'}} \Hom(P', f).   \]
 Since $P$ is $\N$-small the $P'$ actually all have finitely many components (in the colimit we may neglect components which are initial objects). Hence $\Hom(P', f)$ is a weak equivalence being a product of finitely many weak equivalences and, since the colimit is filtered, also $\Hom(P, \amalg f)$ is a weak equivalence. 
\end{proof}

\begin{proof}[Proof of \ref{THEOREMMODELSOBJECTSII}, general case.]
By assumption, the adjunction
\[ \xymatrix{    \mathcal{S} \ar@<2.5pt>[r]^U & \ar@<2.5pt>[l]^{\amalg}  \mathcal{S}^{\amalg}  } \]
satisfies the requirement of Lemma~\ref{LEMMAADJ}, hence we get a weak factorization system $(\Tcof_{\mathcal{S}}, \Fib_{\mathcal{S}})$. 

Lemma~\ref{LEMMATFIB}, 2.\@ applies, hence
it remains to show that $\Tcof_{\mathcal{S}} \subset \mathcal{W}_{\mathcal{S}}$. Take $f \in \Tcof_{\mathcal{S}}$. By Lemma~\ref{LEMMAADJ} it
is a retract of a morphism $f'$ which is in the image under $\amalg$ of $\Tcof_{\mathcal{S}^\amalg}$. Since by  Lemma~\ref{LEMMASCOPROD}  $\amalg$ preserves weak equivalences and those
are closed under retracts $f$ is a weak equivalence as well.

For the extension to diagrams, Theorem~\ref{THEOREMMODELDIAGRAM} reduces to show that the factorization for $(\Tcof_{\mathcal{S}^I}, \Fib_{\mathcal{S}^I})$ exists. This exists by Proposition~\ref{PROPTCOFFIB}, because  $\mathcal{L}^p$ is generated by the morphisms $\emptyset \rightarrow i_! P$ for $i \in I$ and $P \in \mathcal{P}$ and thus by morphisms with $\N$-small domain and codomain. 

The right properness follows from Lemma~\ref{LEMMARIGHTPROPER1}.
\end{proof}

The question arises, of course, how the split-projective structures on $\mathcal{S}^{\Delta^{\op}}$ and $\mathcal{S}^{\amalg, \Delta^{\op}}$ are related, provided they exist:

\begin{SATZ}\label{SATZEXTENSIVEMODELCAT}
Let $\mathcal{S}$ be an extensive category with finite limits, and such that every object is a coproduct of $\N$-small objects. 
Consider the split-projective model categories $\mathcal{S}^{\Delta^{\op}}$ and $\mathcal{S}^{\amalg, \Delta^{\op}}$. 

Then the left Bousfield localization $\mathcal{S}^{\amalg, \Delta^{\op}}_{loc}$  of $\mathcal{S}^{\amalg, \Delta^{\op}}$ at those morphisms that become weak equivalences in $\mathcal{S}^{\Delta^{\op}}$ exists and we have
Quillen adjunctions
\[ \xymatrix{ \mathcal{S}^{\Delta^{\op}} \ar@<2.5pt>[rr]^-U & & \ar@<2.5pt>[ll]^-{\amalg} \mathcal{S}^{\amalg, \Delta^{\op}}_{loc} 
\ar@<2.5pt>[rr]^-{\id} & & \ar@<2.5pt>[ll]^-{\id} \mathcal{S}^{\amalg, \Delta^{\op}} }. \]
The left hand side adjunction is a Quillen equivalence. 
\end{SATZ}

\begin{proof}
We define on $\mathcal{S}^{\amalg, \Delta^{\op}}$ new classes $\mathcal{W}':=\amalg^{-1}(\mathcal{W})$, 
and $\Fib' := {}^\Box (\Cof \cap \mathcal{W}')$, and proceed to show that $(\Cof, \Fib', \mathcal{W}')$ is a model category structure, denoted by $\mathcal{S}^{\amalg, \Delta^{\op}}_{loc}$.
Note that $\mathcal{W}'$ satisfies 2-out-of-3 because $\mathcal{W}$ does. 

We have that $U(\Fib) \subset \Fib'$ because $\amalg$ preserves $\Cof$ and $\mathcal{W}'$ by construction.

We start by constructing a factorization into $\Cof \cap \mathcal{W}'$ and $\Fib'$ as follows. Let $f: X \rightarrow Y$ be a morphism in $\mathcal{S}^{\amalg, \Delta^{\op}}$. We factor 
$\amalg f$  in $\mathcal{S}^{\Delta^{\op}}$ as follows: 
\[ \xymatrix{ \amalg X \ar[rr]^-{\Cof \cap \mathcal{W}}  && Z \ar[rr]^-{\Fib} &&  \amalg Y . } \]
Consider the diagram
\[ \xymatrix{
X \ar[d] \ar[r]^{\Cof}  & X'  \ar[rr]^{\Fib \cap \mathcal{W}} & & Z' \ar[rr]  \ar[d] &&  Y \ar[d] \\
U \amalg X \ar[rrr]_{U(\Cof \cap \mathcal{W})} & & &  UZ \ar[rr]_{U(\Fib)} & & U \amalg Y
} \]
in which the right hand side square is Cartesian and the morphism $X \rightarrow Z'$ has been factored in $\mathcal{S}^{\amalg, \Delta^{\op}}$. Applying $\amalg$ to the left square, we see that 
the morphism $X \rightarrow X'$ is in fact in $\Cof \cap \mathcal{W}'$ (use that $\amalg$ commutes with fibre products and hence $\amalg Z' \cong \amalg U Z \cong Z$). 
The morphism in $\Fib \cap \mathcal{W}$ is obviously in $\Fib'$. Since $U(\Fib) \subset \Fib'$ also the pullback $Z' \rightarrow Y$ is in $\Fib'$. 
Is remains only to show $\Fib' \cap \mathcal{W}' = \Fib \cap \mathcal{W}$. Since $\Cof \Box \Fib \cap \mathcal{W}$ and $\mathcal{W} \subset \mathcal{W}'$ it follows that 
$\Fib \cap \mathcal{W} \subset \Fib' \cap \mathcal{W}'$. It remains to see that $\Cof \Box \Fib' \cap \mathcal{W}'$. Consider a diagram
\[ \xymatrix{
A \ar[r] \ar[dd]_{\Cof} & X \ar@{=}[rr] \ar[rd] \ar[dd]_{\Cof}  & & X \ar[dd]^{\Fib' \cap \mathcal{W}'}  \\  
& &  Y'  \ar[rd]_{\Fib \cap \mathcal{W}} \ar@{.>}[ru]  \\
B \ar[r] &  \lefthalfcap \ar[ru]^{\Cof}\ar[rr]  & & Y
}\]
in which the left square is coCartesian and the morphism $\lefthalfcap \rightarrow Y$ has been factored according to the decoration. 
By 2-out-of-3 the morphism $X \rightarrow Y'$ is in $\Cof \cap \mathcal{W}'$ and thus a lift exists using that $X \rightarrow Y$ is in $\Fib'$. 

The functors $\amalg$ and $U$ are obviously a Quillen adjunction between $\mathcal{S}^{\Delta^{\op}}$ and the localized structure $\mathcal{S}^{\amalg, \Delta^{\op}}_{loc}$. But now 
$\amalg$ and $U$ both preserve weak equivalences and the unit $X \rightarrow U \amalg X$ is in $\mathcal{W}'$ for every $X$. Hence we have a Quillen equivalence. 
\end{proof}

The relation to simplicial pre-sheaves is clarified by the following: 

\begin{PROP}
If $\mathcal{S}$ is small
there is an equivalence of categories with weak equivalences
\[ \xymatrix{ (\mathcal{S}^{\amalg, \Delta^{\op}}, \mathcal{W}) \ar[rr]^-R & &   (\mathcal{SET}^{\mathcal{S}^{\op} \times \Delta^{\op}}, \mathcal{W}) } \]
in which 
\[ R(\{U_i\}_i) (S) := \prod_i \Hom(S, U_i).  \]
Moreover the functor $R$ preserves also cofibrations and fibrations where on the right hand side the projective model category structure is considered.  
\end{PROP}
\begin{proof}
We have the following three facts:

1.\@ $R$ is fully faithful. 

2. $\mathcal{W}$ on the left is precisely the preimage of $\mathcal{W}$ on the right under $R$. 

3. Cofibrant objects on the right (w.r.t.\@ the projective model category structure) are in the essential image of $R$. Indeed, the cofibrant objects in the transported model structure on $\mathcal{SET}^{\mathcal{S}^{\op} \times \Delta^{\op}}$ are, in particular, degree-wise projective w.r.t.\@ the weak factorization system $(\mathcal{L}_{\proj}, \mathcal{R}_{\proj})$ (cf.\@ Proposition~\ref{PROPADHESIVE}). Thus they are coproducts 
of representables because $\mathcal{S}$ is idempotent complete (having finite limits). 

Therefore $R^{-1} Q$ is an inverse (up to natural transformation consisting of weak equivalences) of $R$, where $Q$ is a cofibrant replacement functor and $R^{-1}$ is an inverse of $R$ on the essential image.  

The additional statement follows from Proposition~\ref{PROPFUNCTOR}. 
Indeed, $R$ commutes with finite limits and coproducts. Thus is commutes also with push-outs along morphisms in $\mathcal{L}_{\proj, \spl}$ and transfinite compositions of morphisms in $\mathcal{L}_{\proj, \spl}$. Furthermore we have obviously
\[ R(\mathcal{L}_{\proj, \spl}) \subset \mathcal{L}_{\proj} \quad \text{ and } \quad R(\mathcal{R}_{\proj, \spl}) \subset \mathcal{R}_{\proj}. \]
\end{proof}

\section{Bousfield-Kan formulas and the associated derivator}

Let $(\mathcal{L}, \mathcal{R})$ be a weak factorization system on a category $\mathcal{S}$. Let $I$ be a small category. Recall that $(\mathcal{L}^p, \mathcal{R}^p)$ is defined as the 
pair of classes of morphisms in $\mathcal{S}^I$ in which $\mathcal{R}^p$ consists of those morphisms point-wise in $\mathcal{R}$ and $\mathcal{L}^p:= {}^\Box \mathcal{R}^p$. 
It is not automatic that this constitutes a weak factorization system again. 
However, if $(\mathcal{L}, \mathcal{R})$ is of projective type, then the extension exists automatically (Proposition~\ref{PROPWFSDIAGRAM}). 
Theorem~\ref{THEOREMMODELDIAGRAM} below asserts that, if furthermore $(\mathcal{L}, \mathcal{R})$ yields a model category structure on $\mathcal{S}^{\Delta^{\op}}$,
 then also the projective extension yields a model category structure on $(\mathcal{S}^{\Delta^{\op}})^I$ {\em provided} the relevant factorizations exists. 
 This is in particular the case for the split-projective structure considered in Section~\ref{SECTIONSPLITPROJECTIVE}. 
  
\begin{SATZ}\label{THEOREMMODELDIAGRAM}
Let $\mathcal{S}$ be a category with finite limits and all coproducts, equipped with a fixed weak factorization system $(\mathcal{L}, \mathcal{R})$ such that 
the extensions $(\Cof_{\mathcal{S}}, \Tfib_{\mathcal{S}})$ and $(\Tcof_{\mathcal{S}}, \Fib_{\mathcal{S}})$ (\ref{PARTRANSPORTWFS}) assemble to a model category structure on $\mathcal{S}^{\Delta^{\op}}$.

Assume that 
$(\mathcal{L}^p, \mathcal{R}^p)$
exists and that the corresponding extensions $(\Cof_{\mathcal{S}^I}, \Tfib_{\mathcal{S}^I})$ and 
$(\Tcof_{\mathcal{S}^I}, \Fib_{\mathcal{S}^I})$ exist and are constructed by means of Proposition~\ref{PROPTCOFFIB}. 
Then these assemble to a model category structure on $\mathcal{S}^{I \times \Delta^{\op}}$ with weak equivalences $\mathcal{W}_I$ consisting of those natural transformations which are point-wise in $\mathcal{W}$. Furthermore, we have $(\Cof_{\mathcal{S}^I}, \Tfib_{\mathcal{S}^I}) = (\Cof_{\mathcal{S}}^p, \Tfib_{\mathcal{S}}^p)$ and $(\Tcof_{\mathcal{S}^I}, \Fib_{\mathcal{S}^I}) = (\Tcof_{\mathcal{S}}^p, \Fib_{\mathcal{S}}^p)$ 
In other words, the operations of extension to simplicial objects and forming projective extensions commute. 
\end{SATZ}
 
\begin{proof}
From the description of the pair $(\Cof_{\mathcal{S}^I}, \Tfib_{\mathcal{S}^I})$ one can infer that 
$\Tfib_{\mathcal{S}^I}$ consists of those morphisms that are point-wise in $\Tfib_{\mathcal{S}}$. For that pair therefore the claimed commutativity is clear. 

Define $\mathcal{W}_I$ to be the class of those morphisms that are point-wise in $\mathcal{W}$. 
Assuming that factorizations exist, we have to show
\[ \Tcof_{\mathcal{S}^I} \subset  \mathcal{W}_I   \]
\[ \Tfib_{\mathcal{S}^I} \subset  \mathcal{W}_I   \]
\[ \Fib_{\mathcal{S}^I} \cap  \mathcal{W}_I  \subset \Tfib_{\mathcal{S}^I}. \]
The fourth inclusion follows from the retract argument, as usual. 

We have $\Tcof_{\mathcal{S}^I} \subset  \mathcal{W}_I$ because from the explicit construction of factorizations for $(\Tcof_{\mathcal{S}^I}, \Fib_{\mathcal{S}^I})$, cf.\@ the proof of Proposition~\ref{PROPTCOFFIB}, we see that morphisms in $\Tcof_{\mathcal{S}^I}$ are, in particular, point-wise trivial cofibrations. 
We see immediately that $\Fib_{\mathcal{S}^I}$ consists of those morphisms point-wise in $\Fib_{\mathcal{S}}$, because
a morphism is a fibration if for all  $k, n$ we have
\[ \boxdot \Hom(\Lambda_{n,k} \hookrightarrow \Delta_n, f) \in \mathcal{R}^p \]
and $\boxdot \Hom(\Lambda_{n,k} \hookrightarrow \Delta_n, -)$ is computed  point-wise. From this it follows that 
\[ \Fib_{\mathcal{S}^I} \cap  \mathcal{W}_I = \Tfib_{\mathcal{S}^I} \]
because the equation holds point-wise. 
Since $\Fib_{\mathcal{S}^I}$ and $\Tfib_{\mathcal{S}^I}$ are thus characterized point-wise the extended weak factorization systems are also the projective extensions of the ones on $\mathcal{S}^{\Delta^{\op}}$.
\end{proof}

The goal of the rest of this section is to prove that for $\mathcal{S}^{\Delta^{\op}}$ homotopy limits for (homotopically) finite diagrams and homotopy colimits for arbitrary diagrams exist. Also homotopy Kan extensions exist --- in particular, we get an associated derivator (left derivator on all diagrams and right derivator on homotopically finite diagrams).

To prove a result of this nature there are, as usual, two main strategies: Either one uses model category structures on $(\mathcal{S}^{\Delta^{\op}})^I$, as constructed above, and  
derives the usual limit or colimit functor using the machinery of model categories. Or one uses the simplicial structure to construct homotopy limits and colimits explicitly by establishing  
a Bousfield-Kan formula. We will follow the second strategy here. It gives the results in their most general form without the need for any injective structure on  $(\mathcal{S}^{\Delta^{\op}})^I$ which we did not construct. 
Assume, as usual,  that $\mathcal{S}$ has all colimits and all finite limits.

\begin{SATZ}\label{SATZHOLIM}
Let $\mathcal{S}$ be a category with all coproducts and finite limits. 
Let $\mathcal{S}^{\Delta^{\op}}$ be equipped with a simplicial model category structure as above (with functorial factiorizations). 
Then homotopy colimits exist, i.e.\@ for all diagrams $I$ there is an adjunction
\[ \xymatrix{  (\mathcal{S}^{\Delta^{\op}})^I[\mathcal{W}_I^{-1}] \ar@<-5pt>[rr]_{\hocolim} & &  \ar@<-5pt>[ll]_{p_I^*}   \mathcal{S}^{\Delta^{\op}}[\mathcal{W}^{-1}] }  \]
with $\hocolim$ left adjoint, given for point-wise cofibrant diagrams by the formula
\[ \hocolim X = \int^{I} N(- \times_{/I} I) \otimes X.  \]
(This particular coend is computed by means of coproducts and thus exists, see Remark~\ref{REMCOPROD} below)

Similarly, homotopy limits exist for $I$ homotopically finite, i.e.\@ there is an adjunction
\[ \xymatrix{  (\mathcal{S}^{\Delta^{\op}})^I[\mathcal{W}_I^{-1}] \ar@<5pt>[rr]^{\holim} & &  \ar@<5pt>[ll]^{p_I^*}  \mathcal{S}^{\Delta^{\op}}[\mathcal{W}^{-1}] }  \]
with $\holim$ right adjoint, given for point-wise fibrant diagrams by the formula
\[ \holim X = \int_{I} \Hom_l( N(I \times_{/I} -),  X).  \]
(This particular end is computed by means of finite limits and thus exists, cf.\@ \ref{PARSIMPLICIALSTRUCT}.)
\end{SATZ}

\begin{BEM}\label{REMCOPROD}
From the proof it follows that more explicitly
\[  \left( \int^{I} N(- \times_{/I} I) \otimes X \right)_n = \widetilde{X}_{n,n}, \]
where $\widetilde{X}$ is the bisimplicial object
\[ \widetilde{X}_{n,m} = \coprod_ {i_0 \rightarrow \cdots \rightarrow i_n} X(i_0)_m.  \]
So the coend exists and is computed by means of coproducts.  
\end{BEM}

We first prove:

\begin{LEMMA}\label{LEMMAHOCOLIM}The functor defined for $X \in (\mathcal{S}^{\Delta^{\op}})^{I}$ by
\[ \hocolim_I X := \int^{I} N(- \times_{/I} I) \otimes X \]
maps object-wise weak equivalences between object-wise cofibrant objects to weak equivalences provided that the coend exists, and the functor defined for $X \in (\mathcal{S}^{\Delta^{\op}})^{I}$ by
\[ \holim X := \int_{I}\Hom(N(I \times_{/I} -), X) \]
maps object-wise weak equivalences between object-wise fibrant objects to weak equivalences provided that the end exists.
\end{LEMMA}
\begin{proof}
We concentrate on the case of the coend, the other argument is completely dual. 
If the coend exists, consider the functor associating with $X \in (\mathcal{S}^{\Delta^{\op}})^{I}$ and $Y \in \mathcal{S}^{\Delta^{\op}}$ the set
\begin{eqnarray*} 
F(X,Y) &:=& \Hom_{\mathcal{S}^{\Delta^{\op}}}( \int^{I} N(- \times_{/I} I ) \otimes X, Y)   \\
 &\cong& \Hom_{(\mathcal{SET}^{\Delta^{\op}})^{I^{\op}}}(N(- \times_{/I} I), \Hom_r(X, Y)) 
 \end{eqnarray*}
We have to show that $\boxdot F(f, g)$ is surjective, if $g$ is a trivial fibration and $f$ a point-wise cofibration, 
and also if $g$ is a fibration and $f$ a point-wise trivial cofibration. Indeed, this shows that $X \mapsto \int^{I} N(- \times_{/I} I ) \otimes X$ maps point-wise (trivial) cofibrations to (trivial) cofibrations and thus
preserves all weak equivalences between point-wise cofibrant objects. (In any case, point-wise cofibrant objects form a category of cofibrant objects, which is enough to be able to apply Ken Brown's Lemma.)

We have
\[ \boxdot F(f, g) = \Hom_{(\mathcal{SET}^{\Delta^{\op}})^{I^{\op}}}(N(- \times_{/I} I), \boxdot \Hom_r(f, g)) \]
and if $f$ and $g$ are of the kind above,  $\boxdot \Hom_r(f, g)$ is a point-wise trivial fibration (i.e.\@ a projective trivial fibration). Since the diagram
$i \mapsto N(i \times_{/I} I)$ is projectively cofibrant, the statement follows. 
\end{proof}

\begin{proof}[Proof of Theorem~\ref{SATZHOLIM}]
For the case of homotopy colimits 
consider the diagram
\[ \xymatrix{ I &  \ar[l]_-\iota \int N(I) \ar[r]^-p & \Delta^{\op} } \]
where $\int N(I)$ is the Grothendieck construction applied to the functor $N(I): \Delta^{\op} \rightarrow \mathcal{SET} \subset \Cat$. The functor $p$ is thus an opfibration. 
The functor $\iota$ maps an element $(\Delta_n, \mu: \Delta_n \rightarrow I)$ to $\mu(0)$.
With those functors we have 
\[ \int^{I} N(- \times_{/I} I) \otimes X = \int^{\Delta^{\op}} \delta \otimes (p_! \iota^* X)   \]
in the sense that, if one coend exists then also the other does, and we have a canonical isomorphism. This follows, using Lemma~\ref{LEMMAEND} for the coend, formally from $N(- \times_{/I} I) = \iota_!^{\op} (p^{\op})^* \delta$ where $\delta$ is the canonical cosimplicial simplicial set. 
However, for (bi)simplicial objects $X: \Delta^{\op} \rightarrow \mathcal{S}^{\Delta^{\op}}$ the coend
$\int^{\Delta^{\op}} \delta \otimes X$ always exists and is given by the diagonal simplicial set $(X_{n,n})_n$. This yields the formula given in Remark~\ref{REMCOPROD}. 
Hence the coend exists and by Lemma~\ref{LEMMAHOCOLIM} maps object-wise weak equivalences between object-wise cofibrant objects to weak equivalences. 
We want to apply Theorem~\ref{THEOREMHOCOLIM} to the full subcategory of point-wise cofibrant objects. Observe that then on $\mathcal{S}^{\Delta^{\op}}$ one gets homotopy colimit functors
by composition with a (functorial) cofibrant replacement. 

We use the consideration in \ref{EXAMPLEBK} to establish that $\hocolim_I$ as defined in the statement of the theorem is a calculus of homotopy colimits in the sense of Definition~\ref{DEFFUNCTCALCHOCOLIM}. 
 Note that here there are canonical isomorphisms which can be taken as $\rho$ and $\epsilon$. We proved that 1.\@ the occurring coend exists and 2.\@ the functor $\hocolim_I$ 
 maps point-wise weak equivalences to weak equivalences. Hence we are left to show that $1'$ is a weak equivalence. 
 However, more generally for homotopically finite diagrams, the functor
\[ \hocolim_I X := \int^{I} N(- \times_{/I} I) \otimes X \]
has a right adjoint
\[ Y \mapsto \Hom(N(- \times_{/I} I, Y). \]
This functor maps weak equivalences between fibrant $Y$ to object-wise weak equivalences. Furthermore
the canonical
\[  \Hom(N(- \times_{/I} I), Y) \rightarrow \Hom(\Delta_0, Y) = Y \]
is a point-wise weak equivalence for fibrant $Y$ because $N(i \times_{/I} I)$ is contractible for all $i$.  Therefore $\hocolim_I Q$ (where $Q$ is a functorial cofibrant replacement in $\mathcal{S}^{\Delta^{\op}}$) on the level of homotopy categories is adjoint to a functor which is isomorphic to the constant functor. 
This shows that $\hocolim_I$ is a homotopy colimit for {\em homotopically finite} $I$. In particular, it follows that $1': 1^*X \rightarrow \hocolim_{\Delta_1} X$ is an isomorphism in the homotopy category and thus it is a weak equivalence. 

In the case of homotopy limits, the end in question always exists because we assumed that $I$ is homotopically finite (i.e.\@ $N(I)$ and hence also all $N(I \times_{/I} i)$ are finite simplicial sets, cf.\@ also \ref{PARSIMPLICIALSTRUCT}) and $\mathcal{S}$ is assumed to have finite limits. In this case $\holim_I$ always has a left adjoint whose derived functor is isomorphic to the constant functor. Hence also in this case $1'$ is a weak equivalence,
and we need to invoke Theorem~\ref{THEOREMHOCOLIM} only to get right homotopy Kan extensions\footnote{Alternatively, one can show in this case  that the functor $\alpha_*$ as defined in the statement of Theorem~\ref{THEOREMHOCOLIM} has a left adjoint (before taking homotopy categories) whose derived functor is isomorphic to $\alpha^*$.}.
\end{proof}

\begin{KOR}\label{KORDERIVATOR}
For a functor $\alpha: I \rightarrow J$ between small categories there is an adjunction
\[ \xymatrix{  (\mathcal{S}^{\Delta^{\op}})^I[\mathcal{W}_I^{-1}] \ar@<-5pt>[rr]_{\alpha_!} & &  \ar@<-5pt>[ll]_{\alpha^*}   (\mathcal{S}^{\Delta^{\op}})^J[\mathcal{W}_J^{-1}] }  \]
where the left adjoint (homotopy left Kan extension) is given for point-wise cofibrant objects by the formula
\[ (\alpha_! X)_j =  \hocolim_{I \times_{/J} j} \iota^* X \]
in which $\iota: I \times_{/J} j \rightarrow I$ is the projection and $\hocolim$ denotes the explicit functorial homotopy colimit of Theorem~\ref{SATZHOLIM}.

For a functor $\alpha: I \rightarrow J$ between homotopically finite categories there is an adjunction
\[ \xymatrix{  (\mathcal{S}^{\Delta^{\op}})^I[\mathcal{W}_I^{-1}] \ar@<5pt>[rr]^{\alpha_*} & &  \ar@<5pt>[ll]^{\alpha^*}   (\mathcal{S}^{\Delta^{\op}})^J[\mathcal{W}_J^{-1}] }  \]
where the right adjoint (homotopy right Kan extension) is given for point-wise fibrant objects by the formula
\[ (\alpha_* X)_j = \holim_{j \times_{/J} I} \iota^* X \]
in which $\iota: j \times_{/J} I \rightarrow I$ is the projection, and $\holim$ denotes the explicit functorial homotopy limit of Theorem~\ref{SATZHOLIM}.

In particular the association
\[ I \mapsto \DD(I):= (\mathcal{S}^{\Delta^{\op}})^I[\mathcal{W}_I^{-1}] \]
is a left derivator on all diagrams and a right derivator on homotopically finite diagrams. 
\end{KOR}
\begin{proof}
It was already shown in the proof of Theorem~\ref{SATZHOLIM} that $\hocolim$ (resp.\@ $\holim$) is a transitive calculus of homotopy Kan extensions on point-wise cofibrant (resp.\@ fibrant) objects. Hence,
by Theorem~\ref{THEOREMHOCOLIM} in the appendix, the claimed formulas yield homotopy Kan extensions and the pre-derivator is a left derivator on all small categories (resp.\@ a right derivator
on homotopically finite diagrams). 
\end{proof}

\appendix

\section{Functorial homotopy (co)limits}

In this appendix, we discuss an explicit construction of homotopy (co)limits and Kan extensions, which was used in the proof of Theorem~\ref{SATZHOLIM} and Corollary~\ref{KORDERIVATOR}.
It is of wide applicability --- not only for simplicial categories. The discussion is completely formal and we state the definitions and consequences only for the case of colimits. 
An analogous theory for limits exists and is completely dual. We leave the translation to the reader. 

Consider a category with weak equivalences $(\mathcal{C}, \mathcal{W})$.
Assume that $\mathcal{W}$ is saturated\footnote{i.e.\@ a morphism is in $\mathcal{W}$ if and only if it becomes an isomorphism in $\mathcal{C}[\mathcal{W}^{-1}]$}. Let $\Dia$ be either the category of homotopically finite diagrams (finite directed or, equivalently, inverse categories) or the category of all diagrams.  

\begin{DEF}\label{DEFFUNCTCALCHOCOLIM}
A {\bf functorial\footnote{the adjective ``functorial'' indicates that functoriality in the argument and in the diagram holds in the category $\mathcal{C}$ {\em before} inverting $\mathcal{W}$ } calculus of homotopy colimits} for $(\mathcal{C}, \mathcal{W})$ is the following datum: 
\begin{enumerate} 
\item For $I \in \Dia$ a functor 
\[ \hocolim_I:  \mathcal{C}^I \rightarrow \mathcal{C}. \]
\item For each functor $\alpha: I \rightarrow J$ in $\Dia$ a natural transformation 
\[ \alpha':  \hocolim_I \alpha^* \rightarrow \hocolim_J. \]
\item A natural transformation 
\[ \mathrm{can}: \hocolim_{\{\cdot\}} \rightarrow \id  \]
\end{enumerate} 
satisfying the following axioms: 
\begin{enumerate} 
\item[(HC1)] $\hocolim_I$ maps point-wise weak equivalences to weak equivalences, and $\mathrm{can}$ is an object-wise weak equivalence; 
\item[(HC2)] For all $I \in \Dia$, considering the functor $(\id,1): I \times \Delta_0 \hookrightarrow I \times \Delta_1$, the transformation $(\id,1)'$ is an object-wise weak equivalence;
\item[(HC3)] 
$\id_I'$ is the identity for all $I \in \Dia$, and for $\alpha: I \rightarrow J$ and $\beta: J \rightarrow K$ the diagram
\begin{equation*} \xymatrix{  \hocolim_I (\beta \alpha)^*  \ar[r]^-{(\beta\alpha)'} \ar@{=}[d] &  \hocolim_K \\
 \hocolim_I  \alpha^* \beta^* \ar[r]_{\alpha'} &  \hocolim_J   \beta^*  \ar[u]_{\beta'}
 } 
 \end{equation*} 
 commutes;
\end{enumerate} 
\end{DEF}

 Assume that a category with weak equivalences with a calculus of homotopy colimits is given. 
 
 \begin{LEMMA}\label{LEMMAA2}
 Let $\alpha, \beta: I \rightarrow J$ be functors in $\Dia$ and $\mu: \alpha \Rightarrow \beta$ a natural transformation. Then the two compositions
 \[ \xymatrix{ \hocolim_{I} \alpha^* \ar[r]^{\mu^*}  & \hocolim_{I} \beta^* \ar[r]^{\beta'} & \hocolim_{J(i)} }  \]
 \[ \xymatrix{ \hocolim_{I} \alpha^* \ar[r]^{\alpha'} & \hocolim_{J} }  \]
 are equal in $\Hom(\mathcal{C}^{I}, \mathcal{C})[\mathcal{W}_{\mathcal{C}^I}^{-1}]$. 
 \end{LEMMA}
 
\begin{proof}
Let $M: I \times \Delta_1 \rightarrow J$ be the functor encoding $\mu$ and let $\widetilde{\mu}: M \Rightarrow \beta \circ \pr_1$ be the canonical natural transformation. 
We have the commutative diagrams
 \[ \xymatrix{ 
 \hocolim_{I} \alpha^* \ar[r]^-{(\id,0)'} \ar[d]^{\mu^*} \ar@/^20pt/[rr]^{\alpha'}  & \hocolim_{I \times \Delta_1}  M^* \ar[r]^-{M'} \ar[d]^{\widetilde{\mu}^*}& \hocolim_{J}  \\
 \hocolim_{I} \beta^* \ar[r]^-{(\id,0)'} \ar@/_20pt/[rr]_{\beta'}  & \hocolim_{I \times \Delta_1}  \pr_1^* \beta^*  \ar[r]^-{\beta' \pr_1'} & \hocolim_{J} 
 }   \]
 and
 \[ \xymatrix{ 
 \hocolim_{I} \beta^* \ar[r]^-{(\id,1)'} \ar@{=}[d]  \ar@/^20pt/[rr]^{\beta'}  & \hocolim_{I \times \Delta_1}  M^* \ar[r]^-{M'} \ar[d]^{\widetilde{\mu}^*}& \hocolim_{J}   \\
 \hocolim_{I} \beta^* \ar[r]^-{(\id,1)'}  \ar@/_20pt/[rr]_{\beta'}   & \hocolim_{I \times \Delta_1}  \pr_1^* \beta^*  \ar[r]^-{\beta' \pr_1'} & \hocolim_{J} 
 }   \]
 Both morphisms denoted $(\id,1)'$ are weak equivalences by (HC2) and both horizontal compositions to $\hocolim_{J}$ are the same. Thus we have $M' = \beta' \pr_1' \widetilde{\mu}^*$ in $\Hom(\mathcal{C}^I, \mathcal{C})[\mathcal{W}_{\mathcal{C}^I}^{-1}]$.
 Therefore also $\beta' \mu^*  = \alpha'$ in $\Hom(\mathcal{C}^I, \mathcal{C})[\mathcal{W}_{\mathcal{C}^I}^{-1}]$.
\end{proof}

\begin{LEMMA}\label{LEMMAA3}
 Let $i: \{\cdot\} \hookrightarrow I$ be the inclusion of a final object. Then $i'$ is a weak equivalence. 
\end{LEMMA}
\begin{proof}
We have the natural transformation $\nu: \id \rightarrow i \circ p_I$. Therefore, by Lemma~\ref{LEMMAA2},
\[   \xymatrix{ \hocolim_I   \ar[r]^-{\nu^*} &   \hocolim_I  p_I^* i^* \ar[r]^{p_I'} &  \hocolim_{\{\cdot\}} i^* \ar[r]^{i'}  &  \hocolim_{I} }  \]
is the same as the identity in $\Hom(\mathcal{C}^I, \mathcal{C})[\mathcal{W}_{\mathcal{C}^I}^{-1}]$.  
Furthermore, the composition
\[  \xymatrix{ \hocolim_{\{\cdot\}} i^* \ar[r]^-{i'}  \ar@/_10pt/[rr]_-{i'}  &  \hocolim_I   \ar[r]^-{\nu^*}  &   \hocolim_I  p_I^* i^* \ar[r]^{p_I'}  &  \hocolim_{\{\cdot\}} i^* }   \]
is the identity. Since $\mathcal{W}$ is saturated the statement follows. 
\end{proof}

\begin{PAR}
The goal is to prove: given a functorial calulus of homotopy colimits, we have that
\[ X \mapsto \alpha_! X := ( j \mapsto \hocolim_{I\times_{/J} j} \iota_j^* X ) \]
is a relative Kan extension, i.e.\@ that the above functor defines an adjunction
\[ \xymatrix{ \mathcal{C}^{I}[\mathcal{W}_I^{-1}] \ar@<-3pt>[rr]_{\alpha_!} & & \ar@<-3pt>[ll]_{\alpha^*} \mathcal{C}^{J}[\mathcal{W}_J^{-1}]  }\]
with $\alpha_!$ left adjoint. The functoriality of $\alpha_!$ is given by the following morphism, letting $\mu: j \rightarrow j'$ be a morphism in $J$ and denoting by $\mu$ also
the corresponding functor ``composition'' $I \times_{/J} j \rightarrow I \times_{/J} j'$:
\[ \xymatrix{ \hocolim_{I\times_{/J} j} \iota_j^* X  \ar[r]^{\mu'} & \hocolim_{I\times_{/J} j'} \mu^* \iota_j^* X \ar@{=}[r] &  \hocolim_{I\times_{/J} j'}  \iota_{j'}^* X   } \]
It follows from (HC3) that this indeed defines a functor. 

Using the definition we can write down a (potential) unit and a (potential) counit of the envisaged adjunction:
The counit is given object-wise by the composition
\[ c: \xymatrix{ \hocolim_{I \times_{/J} j} \iota_j^* \alpha^*  \ar[r]^-{\nu^*} & \hocolim_{I \times_{/J} j} p_{I \times_{/J} j}^* j^* \ar[r]^-{p_{I \times_{/J} j}'} & \hocolim_{\{ \cdot \}} j^* \ar[r]^-{ \mathrm{can} } & j^*  } \]
and the unit is given by the composition
\[ u: \xymatrix{  \id & \ar[l]_-{\rho}  \id_{I,!} \ar[r]^-{\kappa} & \alpha^* \alpha_!  } \]
where $\kappa_i$ is given by
\[ \xymatrix{  \hocolim_{I\times_{/I} i} \iota_i^* \ar[r]^-{\widetilde{\alpha}'} & \hocolim_{I\times_{/J} \alpha(i)} \iota_{\alpha(i)}^* }    \]
where $\widetilde{\alpha}: I\times_{/I} i \rightarrow I\times_{/J} \alpha(i)$ denotes the obvious functor induced by $\alpha$,
and $\rho_i$ is given by the composition
\[ \xymatrix{  \hocolim_{I\times_{/I} i} \iota_i^* \ar[r]^-{\nu*} &  \hocolim_{I\times_{/I} i} p_{I \times_{/I} i}^* i^*  \ar[rr]^-{p_{I\times_{/I} i}'} &&  \hocolim_{\{\cdot\}} i^* \ar[r]^-{\mathrm{can}} & i^*.    }   \]
It follows from (HC3) that $c$, $\kappa$ and $\rho$ are indeed natural in $i$. 
Furthermore the diagram
\[ \xymatrix{ \hocolim_{\{\cdot\}} i^* \ar[r]^-{(\id_i)'} \ar@{=}[d] & \hocolim_{I\times_{/I} i} \iota_I^* \ar[d]^{\nu^*}   \\
 \hocolim_{\{\cdot\}} i^* \ar[r]^-{(\id_i)'} \ar@/_20pt/[rrr]_{\id} & \hocolim_{I\times_{/I} i} p_{I \times_{/I} i}^* i^* \ar[rr]^-{p_{I\times_{/I} i}'} & &    \hocolim_{\{\cdot\}} i^*    } \]
 is commutative. As $(\id_i)'$ is a weak equivalence by Lemma~\ref{LEMMAA3} it follows, using (HC1), that also $\rho_i$ is a weak equivalence.
 Therefore $\rho$ is a weak equivalence and can be inverted, and thus the definition of unit makes sense. 
\end{PAR}

 \begin{PAR}To get a valid adjunction, we have to verify the unit/counit equations. Without imposing further axioms they probably will not hold. 
 Let $I \in \Dia$, consider a functor $F: I \rightarrow \Dia$, and let $X \in \mathcal{C}^{\int F}$. It follows from (HC3) that the association 
 \[ i \mapsto \hocolim_{F(i)} X|_{F(i)} \]
 is functorial in $i$. 
 \end{PAR}
 
 \begin{DEF}\label{DEFFUNCTCALCHOCOLIMTRANSITIVE}
 A {\bf transitive} functorial calculus of homotopy colimits for $(\mathcal{C}, \mathcal{W})$ is as in Definition~\ref{DEFFUNCTCALCHOCOLIM} with
 \begin{enumerate} 
\item for each $I \in \Dia$ and functor $F: I \rightarrow \Dia$  a natural transformation
\[ \Xi_{F}: \hocolim_I \hocolim_{F(-)} \rightarrow \hocolim_{\int F}  \]
between functors $\mathcal{C}^{\int F} \rightarrow \mathcal{C}$,
\end{enumerate} 
satisfying in addition: 
\begin{enumerate} 
 \item[(HC4)] For any functor $\alpha: I \rightarrow J$ in $\Dia$ and a functor $F: J \rightarrow \Dia$ the following diagram commutes:
 \[ \xymatrix{ \hocolim_{I}  \alpha^* \hocolim_{F(\alpha(-))} \ar[d]_{\Xi_{\alpha^* F}} \ar[rr]^-{\alpha' } & & \hocolim_{J} \hocolim_{F(-)}  \ar[d]^{\Xi_{F}} \\
\hocolim_{\int \alpha^* F} (\alpha')^*     \ar[rr]_{(\alpha')'}  & &  \hocolim_{\int F}  } \]
Similarly, for $I \in \Dia$ and a natural transformation $\mu: F \Rightarrow G$ of functors $F, G: I \rightarrow \Dia$, the following diagram commutes:
 \[ \xymatrix{ \hocolim_{I} \hocolim_{F(-)} \mu(-)^*  \ar[d]_{\Xi_F} \ar[rr]^-{\hocolim_I \mu(-)'} & & \hocolim_{I} \hocolim_{G(-)}  \ar[d]^{\Xi_G} \\
 \hocolim_{\int F}  (\int \mu)^* \ar[rr]_-{(\int \mu)'}  & & \hocolim_{\int G}   } \]
\item[(HC5)] The morphisms  $\Xi$ and $\hocolim_{I} \circ \mathrm{can}$ 
 \[ \xymatrix{ \hocolim_{I} \hocolim_{\{\cdot\}}  \ar[r] & \hocolim_{I}  } \]
 are equal in $\Hom(\mathcal{C}^I, \mathcal{C})[\mathcal{W}_{\mathcal{C}^I}^{-1}]$ and similarly in the other order.  
\end{enumerate} 
\end{DEF}
 
 A transitive functorial calculus of homotopy colimits is a constructive device (without the need for model category structures) to obtain a left derivator: 
 \begin{SATZ}\label{THEOREMHOCOLIM}
Let $(\mathcal{C}, \mathcal{W})$ be a category with weak equivalences, assume that $\mathcal{W}$ is saturated, and let $I \mapsto \hocolim_I$ be a transitive functorial calculus of homotopy colimits. Then
$\hocolim_I$ is a homotopy colimit and 
\[ X \mapsto \alpha_! X := ( j \mapsto \hocolim_{I\times_{/J} j} \iota_j^* X ) \]
is a relative Kan extension satisfying Kan's formula. 
In other words, the pre-derivator
\[ \DD: I \mapsto \mathcal{C}^I [\mathcal{W}_I^{-1}]  \]
is a left derivator with domain $\Dia$ (in which the categories $\DD(I)$ are not necessarily locally small).
\end{SATZ}
 
 \begin{BEISPIEL}[Bousfield-Kan] \label{EXAMPLEBK}
 Let $(\mathcal{C}, \mathcal{W})$ be a category with weak equivalences with a functor\footnote{which need not necessarily be part of the structure of a simplicial category}
 \[  \otimes: \mathcal{SET}^{\Delta^{\op}} \times \mathcal{C} \rightarrow \mathcal{C}   \]
 (respectively restricted to finite simplicial sets, if $\Dia$ is the category of homotopically finite diagrams) with an ``Eilenberg-Zilber'' weak equivalence:
 \[ \rho_{Y,X,C}: Y \otimes (X \otimes C) \rightarrow (Y \times X) \otimes C  \]
 functorial in $Y, X$ and $C$, 
 and a weak equivalence
 \[  \epsilon_C: \Delta_0 \otimes C \rightarrow C \]
 functorial in $C$ such that $\rho$ and $\epsilon$ are compatible in the sense that
 \[ \xymatrix{ Y \otimes (\Delta_0 \otimes C) \ar@<3pt>[rr]^-{\rho_{Y, \Delta_0, C}} \ar@<-3pt>[rr]_-{Y \otimes (\epsilon_C)} & & Y \otimes C } \]
are equal in $\Fun(\mathcal{C}, \mathcal{C})[\mathcal{W}^{-1}_{\mathcal{C}}]$ and similarly in the other order.
 Then 
 \[ \hocolim_I(X) := \int^I N(- \times_{/I} I) \otimes X \]
 is a transitive calculus of homotopy colimits on $\Dia$ as in Definition~\ref{DEFFUNCTCALCHOCOLIM}, provided that 
 \begin{enumerate}
 \item the coend exists,
 \item $\hocolim_I$ maps object-wise weak equivalences to weak equivalences, and
 \item $(\id, 1)'$ (defined in the proof) is a weak equivalence, where $(\id, 1): I \times \Delta_0 \rightarrow I \times \Delta_1$. 
 (If $\rho$ consists of isomorphisms, then it suffices to check this for $I = \Delta_0$.)
 \end{enumerate}
 \end{BEISPIEL}
 \begin{proof}
The morphism $\alpha'$ is given by
 \[ \alpha': \int^I N(- \times_{/I} I) \otimes (\alpha^* X) \rightarrow \int^I ((\alpha^{\op})^* N(- \times_{/J} J)) \otimes (\alpha^* X) \rightarrow \int^J N(- \times_{/J} J) \otimes X  \]
 using the functoriality of the coend in $I$. 
 The morphism $\mathrm{can}$ is given by
 \[ \mathrm{can}: \xymatrix{  \int_{\{\cdot\}} N(\cdot) \otimes X = \Delta_0 \otimes X \ar[r]^-{\epsilon_X} & X  }   \]
 and the morphism $\Xi_F$ (transitivity) for a functor $F: I \rightarrow \Dia$ and $X: \int F \rightarrow \mathcal{C}$ is given by the following construction: 
 First we have, by definition, 
 \begin{eqnarray*}
  \int^I N(- \times_{/I} I) \otimes \int^{F(-)} N(- \times_{/F(-)} F(-)) \otimes X 
  &=& \colim_{\tw I \times_{I} \int \tw F} F(i \rightarrow i', \underbrace{j \rightarrow j'}_{\in \Mor(F(i))}) \\
   \int_{\int F} N(- \times_{/\int F} \int F) \otimes X &=& \colim_{\tw (\int F)} G(\alpha: i \rightarrow i', j \in F(i), \underbrace{F(\alpha)(j) \rightarrow j'}_{\in \Mor(F(i'))})
 \end{eqnarray*}
 with 
\begin{eqnarray*}
 F(i \rightarrow i', j \rightarrow j') &:=& N(i' \times_{/I} I) \otimes (N(j' \times_{/F(i)} F(i)) \otimes X(i,j))  \\
 G(\alpha: i \rightarrow i', j, F(\alpha)(j) \rightarrow j') &:=& N((i',j') \times_{/\int F} \int F) \otimes X(i,j). 
\end{eqnarray*}
 Note that neither $F$, nor $G$, depend on the morphisms, as always in the construction of coend as colimit over the twisted arrow category. 
 Consider the functor
 \begin{eqnarray*}
  \Omega: \tw I \times_{I} \int \tw F &\rightarrow& \tw (\int F)  \\
  (\alpha: i \rightarrow i', \underbrace{j \rightarrow j'}_{\in \Mor(F(i))}) & \mapsto &(\alpha:  i  \rightarrow i', j, F(\alpha)(j) \rightarrow F(\alpha)(j')).   
  \end{eqnarray*}
We define $\Xi_F$ by the morphism
 \[ F \rightarrow \Omega^* G \]
 given at an object $(\alpha: i \rightarrow i', j \rightarrow j') \in \tw I \times_{I} \int \tw F$ by the composition
 \begin{gather*} \xymatrix{ 
 & N(i' \times_{/I} I) \otimes \left(N(j' \times_{/F(i)} F(i)) \otimes X(i,j)\right)  
 \ar[r]^{\rho}  & \left(N((i' \times_{/I} I) \times (j' \times_{/F(i)} F(i)))\right) \otimes X(i,j)  } \\
 \xymatrix{ \ar[r] & N( (i', F(\alpha)(j')) \times_{/\int F} \int F) \otimes X(i,j) } \end{gather*}
where the first map is the Eilenberg-Zilber map and the second is induced by the following functor
\begin{eqnarray*}
  (i' \times_{/I} I) \times (j' \times_{/F(i)} F(i)) &\rightarrow& (i',F(\alpha)(j')) \times_{/\int F} \int F \\
  (\beta: i' \rightarrow i'_2, j' \rightarrow j_2') &\mapsto& (i' \rightarrow i'_2, F(\alpha)(j'), F(\beta\alpha)(j') \rightarrow F(\beta\alpha)(j_2')).
\end{eqnarray*}
This functor is independent of $\alpha$ and an isomorphism, if $F$ is constant. In that case also $\Omega$ is an isomorphism. By consequence $\Xi_F$ is an isomorphism if $F$ is constant and the Eilenberg-Zilber map is an isomorphism. 
The verification of the axioms (HC1)--(HC5) is left to the reader. 
 \end{proof}

\begin{LEMMA}For a transitive functorial calculus of homotopy colimits, (HC2) is, in the presence of the the other axioms, equivalent to:
\begin{enumerate} 
\item[(HC2')] $\Xi_{I,\Delta_1}: \hocolim I \hocolim \Delta_1 \rightarrow \hocolim_{I \times \Delta_1}$ is an object-wise weak equivalence, and for $1: \Delta_0 \hookrightarrow \Delta_1$, $1'$ is an object-wise weak equivalence.
\end{enumerate} 
\end{LEMMA}
\begin{proof}
 We have a commutative diagram
 \[ \xymatrix{ \hocolim_I \hocolim_{\Delta_0} \ar[r]^-{1'} \ar[d]_{\Xi_{I,\Delta_0}} & \hocolim_I \hocolim_{\Delta_1}  \ar[d]^{\Xi_{I,\Delta_1}} \\
  \hocolim_I \ar[r]_-{(\id,1)'} & \hocolim_{I \times \Delta_1}
 } \]
By (HC1) and (HC5) the left vertical morphism is an object-wise weak equivalence. If (HC2) holds then the top and bottom horizontal morphisms are object-wise weak equivalences for all $I$.
Hence also the right vertical morphism is an object-wise weak equivalence for all $I$ which implies (HC2'). If (HC2') holds then the right vertical and top horizontal morphisms are weak equivalences for all $I$. Thus
also the bottom horizontal morphism is an object-wise weak equivalence, i.e.\@ (HC2) holds. 
\end{proof}

\begin{proof}[Proof of Theorem~\ref{THEOREMHOCOLIM}.]
(Der1) is clear and (Der2) follows from the saturatedness of $\mathcal{W}$.

We have to check the unit/counit equations.
The first equation asserts that the composition
\[ \xymatrix{ \alpha^* \ar[rr]^-{u \alpha^*} & & \alpha^* \alpha_! \alpha^* \ar[rr]^-{\alpha^* c} & &  \alpha^*  }\]
 is object-wise the identity in $\Fun(I, \mathcal{C})[\mathcal{W}^{-1}_I]$. We will use the abbreviation 
\[ \hocolimshort_{I} := \hocolim_{I}. \]
 Inserting the definitions, it suffices to show that the composition
\[ \footnotesize \xymatrix{ \hocolimshort_{\{\cdot\}} \alpha(i)^* & \ar[l]_-{p_{I \times_{/I} i}'} \hocolimshort_{I \times_{/I} i} p_{I \times_{/I} i}^* \alpha(i)^* \ar@/_40pt/[rrr]_-{\widetilde{\alpha}'}  & \ar[l]_-{\nu^*} \hocolimshort_{I \times_{/I} i} \iota_i^* \alpha^* \ar[r]^-{\widetilde{\alpha}'} & \hocolimshort_{I \times_{/J} \alpha(i)} \iota_{\alpha(i)}^* \alpha^* \ar[r]^{\nu^*} & \hocolimshort_{I \times_{/J} \alpha(i)} p_{I \times_{/J} \alpha(i)}^* \alpha(i)^* \ar[r]^-{p_{I \times_{/J} j}'} & \hocolimshort_{\{\cdot\}} \alpha(i)^*   } \]
is the identity for all $i$. This follows from the functoriality of $\hocolim$ and (HC3). The second unit/counit equation asserts that the composition
\[ \xymatrix{  \alpha_! \ar[rr]^-{\alpha_! u} & &  \alpha_! \alpha^* \alpha_! \ar[rr]^-{c \alpha_!} & &  \alpha_!  }\]
is object-wise the identity in $\Fun(J, \mathcal{C})[\mathcal{W}^{-1}_J]$. Inserting the definitions, it suffices to show that the composition in the top row of Figure~\ref{fig:unitcounit} is the identity in $\Fun(\mathcal{C}^I, \mathcal{C})[\mathcal{W}_{\mathcal{C}^I}^{-1}]$. 
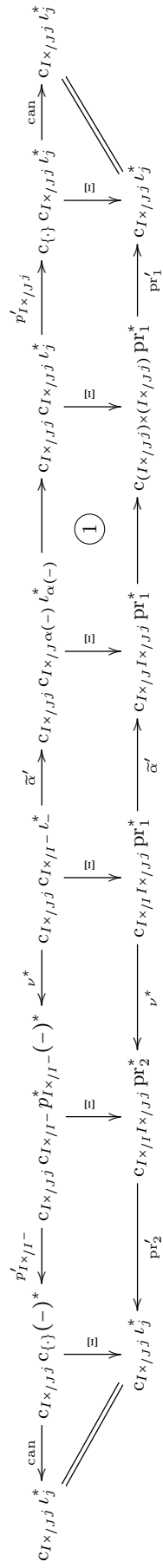
\begin{sidewaysfigure}
    \centering
\[ \footnotesize
 \xymatrix{ 
\hocolimshort_{I \times_{/J} j} \iota_j^* \ar@{=}[rd] &
 \ar[l]_-{\mathrm{can}} \hocolimshort_{I \times_{/J} j} \hocolimshort_{\{\cdot\}} (-)^* \ar[d]^-{\Xi} & 
 \ar[l]_-{p_{I \times_{/I} -}'} \hocolimshort_{I \times_{/J} j} \hocolimshort_{I \times_{/I} -} p_{I \times_{/I} -}^* (-)^*  \ar[d]^{\Xi}  & 
 \ar[l]_-{\nu^*} \hocolimshort_{I \times_{/J} j} \hocolimshort_{I \times_{/I} -} \iota_{-}^*  \ar[r]^-{\widetilde{\alpha}'}  \ar[d]^{\Xi}  & 
 \hocolimshort_{I \times_{/J} j} \hocolimshort_{I \times_{/J} \alpha(-)} \iota_{\alpha(-)}^* \ar@{}[rd]|{\ccircled{1}}  \ar[r]  \ar[d]^{\Xi}  & 
 \hocolimshort_{I \times_{/J} j} \hocolimshort_{I \times_{/J} j} \iota_j^* \ar[r]^-{p_{I \times_{/J} j}'}  \ar[d]^{\Xi}  & 
 \hocolimshort_{\{\cdot\}} \hocolimshort_{I \times_{/J} j} \iota_j^* \ar[r]^-{\mathrm{can}} \ar[d]^{\Xi}  &
  \hocolimshort_{I \times_{/J} j} \iota_j^* \ar@{=}[ld]  \\
& 
 \hocolimshort_{I \times_{/J} j} \iota_j^* &
 \hocolimshort_{I \times_{/I} I \times_{/J} j} \pr_2^*  \ar[l]^-{\pr_{2}'}  &
 \hocolimshort_{I \times_{/I} I \times_{/J} j} \pr_1^* \ar[l]^-{\nu^*}  \ar[r]_-{\widetilde{\alpha}'}  & 
 \hocolimshort_{I \times_{/J} I  \times_{/J} j} \pr_1^*  \ar[r]  & 
 \hocolimshort_{(I \times_{/J} j) \times (I  \times_{/J} j)} \pr_1^*  \ar[r]_-{\pr_{1}'}  & 
 \hocolimshort_{I \times_{/J} j}  \iota_j^*  
}  \]
    \caption{The second unit/counit equation}
    \label{fig:unitcounit}
\end{sidewaysfigure}
It follows from (HC4) that all squares in Figure~\ref{fig:unitcounit} are commutative. We illustrate this for the perhaps most involved square marked as $\ccircled{1}$ in the figure:
\[
 \xymatrix{  
 \hocolimshort_{I \times_{/J} j} \hocolimshort_{I \times_{/J} \alpha(-)} \iota_{\alpha(-)}^*  \ar[rr]^-{\hocolimshort_{I \times_{/J} j} (\widetilde{\beta}')_{\beta} }  \ar[d]_{\Xi}  & & 
 \hocolimshort_{I \times_{/J} j} \hocolimshort_{I \times_{/J} j} \iota_j^* \ar[d]^{\Xi} \\
 \hocolimshort_{I \times_{/J} I  \times_{/J} j} \pr_1^*  \ar[rr]_-{(\int \mu)'}  & & 
 \hocolimshort_{(I \times_{/J} j) \times (I  \times_{/J} j)} \pr_1^*   
}
\]
Its commutativity is (HC4) for the morphisms of functors $\mu: F \Rightarrow G$, where $F: I \times_{/J} j \rightarrow \Dia$ is the functor mapping
$(i, \beta: \alpha(i) \rightarrow j)$ to $I \times_{/J} \alpha(i)$ and $G: I \times_{/J} j \rightarrow \Dia$ is the constant functor
$I \times_{/J} j$ and $\mu: F \Rightarrow G$ at $(i, \beta: \alpha(i) \rightarrow j)$ is given by the functor ``composition'' $\widetilde{\beta}: I \times_{/J} \alpha(i) \rightarrow I \times_{/J} j$.

Since all squares in Figure~\ref{fig:unitcounit} commute in $\Hom(\mathcal{C}^I, \mathcal{C})[\mathcal{W}_{\mathcal{C}^I}^{-1}]$, 
 we are reduced to show that the two morphisms (composition all the way to the left, and to the right, respectively)
\[ \hocolimshort_{I \times_{/I} I \times_{/J} j} \pr_1^* \rightarrow \hocolimshort_{I \times_{/J} j} \iota_j^* \] 
 are the same in $\Hom(\mathcal{C}^{I}, \mathcal{C})[\mathcal{W}_{\mathcal{C}^{I}}^{-1}]$. By Lemma~\ref{LEMMAA2} applied to the canonical natural transformation
 \[ \pr_{13} \Rightarrow \pr_{23}:  I \times_{/I} I \times_{/J} j \rightarrow I \times_{/J} j \]
 this is indeed the case.
This shows (Der3), and (Der4) holds by construction.  
\end{proof}

At several places the  following folklore Lemma was used: 
\begin{LEMMA} \label{LEMMAEND}
Let $F: I^{\op} \times J \rightarrow \mathcal{C}$ be a functor, and $\alpha: I \rightarrow J$. Then we have
\[ \int_J (\alpha^{\op}, \id)_* F \cong \int_I (\id, \alpha)^* F    \]
in the sense that, if one end exists, so does the other, and we have a canonical isomorphism. Similarly for the coend: 
\[ \int^J (\alpha^{\op}, \id)_! F \cong \int^I (\id, \alpha)^* F.    \]
\end{LEMMA}

\newpage
\bibliographystyle{abbrvnat}
\bibliography{6fu}

\end{document}